\documentclass{amsart}

\usepackage{xr}
\usepackage{hyperref}

\usepackage{amsmath}
\usepackage{amssymb}
\usepackage{amsthm}
\usepackage{xcolor}
\usepackage{mathtools}
\usepackage{cleveref}
\usepackage{subcaption}

\usepackage{algorithm}
\usepackage{algpseudocode}

\newcommand{\SO}{\operatorname{SO}(3)}
\newcommand{\R}{\mathbb{R}}
\newcommand{\ba}{\mathbf{a}}
\newcommand{\bx}{\mathbf{x}}
\newcommand{\by}{\mathbf{y}}
\newcommand{\bz}{\mathbf{z}}
\newcommand{\bb}{\mathbf{b}}

\DeclareMathOperator*{\argmin}{argmin \,}

\newtheorem{theorem}{Theorem}
\newtheorem{assumption}{Assumption}

\newtheorem{lemma}[theorem]{Lemma}
\newtheorem{corollary}[theorem]{Corollary}
\newtheorem{remark}[theorem]{Remark}
\newtheorem{definition}[theorem]{Definition}

\title[Autocorrelation analysis for cryo-EM with sparsity constraints]{Autocorrelation analysis for cryo-EM with sparsity constraints: Improved sample complexity and projection-based algorithms}

\author[T. Bendory]{Tamir Bendory}
\email{bendory@tauex.tau.ac.il}
\author[Y. Khoo]{Yuehaw Khoo}
\email{ykhoo@uchicago.edu}
\author[J. Kileel]{Joe Kileel} 
\email{jkileel@math.utexas.edu}
\author[O. Mickelin]{Oscar Mickelin} \email{hm6655@princeton.edu}
\author[A. Singer]{Amit Singer}
\email{amits@math.princeton.edu}

\thanks{
TB is supported in part by the ISF grant no. 1924/21, the BSF grant no. 2020159,  and the NSF-BSF grant no. 2019752.
JK is supported in part by start-up grants from the College of Natural Sciences and Oden Institute for Computational Engineering and Sciences at UT Austin.
AS is supported in part by AFOSR FA9550-20-1-0266, the Simons Foundation Math+X Investigator Award, NSF BIGDATA IIS-1837992, NSF DMS-2009753, and NIH/NIGMS 1R01GM136780-01. 
JK thanks João M. Pereira for useful conversations. We would like to thank the editor and two anonymous reviewers whose suggestions improved the paper.}

\begin{document}

\maketitle

\begin{abstract}
	The number of noisy images required for molecular reconstruction in single-particle cryo-electron microscopy (cryo-EM) is governed by the autocorrelations of the observed, randomly-oriented, noisy projection images. 
	In this work, we 
	consider the effect of imposing  sparsity priors on the molecule. We use techniques from signal processing, optimization, and applied algebraic geometry to obtain new theoretical and computational contributions for this challenging non-linear inverse problem with sparsity constraints. 
	We
	prove that molecular structures modeled as
 sums of Gaussians are uniquely determined by the second-order autocorrelation of their projection images, implying that the sample complexity is proportional to the square of the variance of the noise. 
	This theory improves upon the non-sparse case, 
	where the third-order autocorrelation is required for uniformly-oriented particle images and the sample complexity scales with the cube of the noise variance.  
	Furthermore, we build a computational framework to reconstruct molecular structures which are sparse in the wavelet basis.
This method combines the sparse representation for the molecule with projection-based techniques used for phase retrieval in X-ray crystallography.
\end{abstract}

Sparsity is a ubiquitous prior in many linear inverse problems, including regression~\cite{tibshirani1996regression,goodfellow2016deep}, compressed sensing~\cite{donoho2006compressed,candes2006robust,eldar2012compressed}, and various image processing applications~\cite{elad2010sparse}, to name a few. 
While sparse priors are also used for non-linear inverse problems, 
their applicability and theoretical foundations are  limited to a few specific (usually linear and quadratic) models, e.g.,~\cite{cai2016optimal,chi2016guaranteed, elser2018benchmark,zhang2019structured}.
Motivated by single-particle cryo-EM---an imaging technology  for determining the 3-D structure of biological molecules---this paper 
uses modern techniques from signal processing, optimization, and applied algebraic geometry to provide new theoretical analysis and computational methods for a challenging non-linear inverse problem with sparsity constraints.

Cryo-EM has garnered increasing interest in the past decade due to a series of technological and algorithmic
breakthroughs, driving a striking improvement in the obtainable resolution, up to the level where
individual atoms can be distinguished.
This has in turn opened new scientific horizons  and led 
to many biological discoveries, e.g.,~\cite{kuhlbrandt2014resolution,bai2015cryo,callaway2020revolutionary}.

In a cryo-EM experiment, a solution containing molecules to be imaged is rapidly frozen into a thin ice layer, which is then placed in an electron microscope. 
Next, the  microscope acquires an image, called micrograph, which contains multiple
2-D tomographic projection images of the molecules. 
The 3-D orientations
 of individual projection images are unknown and random.
 To avoid damaging the samples, the electron dose must be kept low, resulting in a  low signal-to-noise ratio (SNR). 
 The cryo-EM computational problem is reconstructing the 3-D molecular structure 
 from these projection images~\cite{frank2006three}.

 The renewed interest in cryo-EM   led to a thorough investigation of its mathematical and statistical foundations~\cite{singer2018mathematics,bendory2020single}.  
In particular, a crucial challenge from a statistical perspective is understanding the \emph{sample complexity} of cryo-EM, i.e., the number of images that are required 
to obtain accurate reconstructions.
A remarkable result revealed an intimate connection between the sample complexity of cryo-EM (and related statistical models) and the \emph{method of moments} in the low SNR regime. 
If the distribution of the 3-D rotations is uniform, the method of moments reduces to autocorrelation analysis. 
In particular, 
it was shown that if $d$ is the lowest degree moment of the observations (i.e., the randomly oriented tomographic projections) that determines the molecular structure uniquely, 
a necessary condition for recovery is $n=\omega(\sigma^{2d}),$ namely, $n/\sigma^{2d}\to\infty$ as $n\to\infty$) where $\sigma^2$ is the variance of the noise~\cite{perry2019sample,abbe2018estimation}.
Specifically,  if the distribution of rotations is uniform, then   the third-order autocorrelation is the lowest order autocorrelation that  determines a  generic 3-D structure, implying a sample complexity of  $n=\omega(\sigma^6)$~\cite{bandeira2017estimation}. 
This agrees with long-standing empirical evidence~\cite{sigworth1998maximum}.

Autocorrelation analysis was  first introduced to cryo-EM by Zvi Kam more than 40 years ago~\cite{kam1980reconstruction}.
 Kam showed that the second-order autocorrelation of the projection images (which can be estimated with $n=\omega(\sigma^4)$ observations) determines the 3-D structure up to a set of orthogonal matrices, under the assumption that the rotations are drawn from a uniform distribution.
 A few methods have been proposed  to resolve the missing orthogonal matrices based on typically  unavailable  side information, such as homologous models with known structure, or a few  clean projections images, in order to 
construct ab initio models~\cite{bhamre2015orthogonal,levin20183d,huang2022orthogonal}.
These ab initio models can then be refined using expectation-maximization: the prevalent algorithmic framework for the cryo-EM reconstruction problem that aims to  maximize the non-convex  posterior distribution~\cite{sigworth1998maximum,scheres2012relion,punjani2017cryosparc}.
In addition, it was recently shown that if the distribution of  rotations is non-uniform,
then there is at most a finite list of structures that agree with the second moment of the observations~\cite{sharon2020method}.
Techniques that are inspired by Kam's method were also proposed
as a solution to the molecular reconstruction problem in X-ray free-electron lasers (XFEL), which, akin to the reconstruction problem in cryo-EM, involves recovering a 3-D structure from its randomly oriented diffraction images~\cite{saldin2011reconstructing,kurta2013solution,donatelli2015iterative,kurta2017correlations}. 
In contrast to cryo-EM, in XFEL the rotations are more likely to be uniformly distributed for particles with nearly uniform dimensions and the reconstruction problem is more involved since the measurements consist of the magnitudes of Fourier coefficients without their phases.

The first contribution of this paper is proving  that if the sought-for molecular structure can be described as a sparse combination of Gaussian functions,
then
the structure can be determined uniquely from the 
 second-order autocorrelation of the observations, even if the rotations are distributed uniformly.  This  eliminates the orthogonal matrix ambiguity in Kam's original paper~\cite{kam1980reconstruction}.
This result is the first theoretical guarantee for unique recovery from the second moment. 
 It combines the sparsity assumption with proof techniques from real algebraic  geometry,
substantially reducing the sample complexity of the cryo-EM reconstruction problem from $n=\omega(\sigma^6)$ to $n=\omega(\sigma^4)$.
The argument is constructive in  the sense that it provides a polynomial-time recovery algorithm. However, the said algorithm is tailored to the specific model of point mass functions.  
It is not well-suited to data discretized into pixels or voxels because it hinges on the ability to cluster points in the support of the second moment into certain distinct components, see Remark~\ref{rem:not-practical}. (Accurate clustering becomes difficult when data is discretized and the components are close to each other.)

The second contribution of this paper is 
 a practical algorithm, fusing the second-order autocorrelation of the projections with a sparse representation of the molecular structure, and requiring only $n=\omega(\sigma^4)$ observations. 
The algorithm builds on the realization that a typical 3-D structure 
can be represented by only a few coefficients in a suitable basis.
Similar sparsity assumptions have been leveraged in a wide variety of scientific and engineering applications, including compressed sensing~\cite{donoho2006compressed,candes2006robust,eldar2012compressed}, image processing~\cite{elad2010sparse}, and phase retrieval \cite{cai2016optimal,chi2016guaranteed, jaganathan2017sparse, elser2018benchmark,zhang2019structured}, to name a few.
In cryo-EM, there have been several attempts to represent the 3-D structure as either a sparse mixture of Gaussians~\cite{zehni20203d,joubert2015bayesian, jonic2015coarse,kawabata2018gaussian,zhong2021exploring,chen2021deep,rosenbaum2021inferring} or using alternative bases~\cite{vonesch2011fast}. In particular, in settings with sufficiently high SNR where it is possible to identify the Gaussian mixtures within individual projection images, the inverse problem simplifies and can be solved \cite{panaretos2009random,panaretos2011sparse}. 
Yet, the sparsity property  has still not been fully harnessed to represent and recover 3-D molecular structures, and it is not part of the standard computational pipeline of cryo-EM. 
 
The technique is based on a new connection between Kam's theory for cryo-EM and the crystallographic phase retrieval problem---recovering a sparse signal from its Fourier magnitudes. 
In particular, we adapt projection-based algorithms that were designed for the crystallographic phase retrieval problem to the cryo-EM setting.
These algorithms were extensively validated on experimental X-ray crystallography datasets by prior researchers, see for example~\cite{fienup1982phase,elser2003phase,luke2004relaxed,elser2007searching,elser2018benchmark}.
Here, we demonstrate on simulated data that they are also useful in constructing ab initio models in cryo-EM.  They can be used to mitigate computational and model bias issues  associated with the non-convexity of the cryo-EM reconstruction problem~\cite{singer2011three,punjani2017cryosparc,greenberg2017common}.
 This new computational approach opens the door to merging more aspects of the phase retrieval and cryo-EM fields in future work. 

The rest of the paper is organized as follows.
In Section~\ref{sec:preliminary_definitions}, we provide background on the reconstruction problem in cryo-EM, the method of moments, Kam's theory, and the crystallographic phase retrieval problem. In Section~\ref{sec:theoretical}, we prove that a structure composed of an ensemble of ideal point masses subject to uniform rotations can be recovered from the second order autocorrelation, implying a sample complexity of $n=\omega(\sigma^4)$. Section~\ref{sec:algorithm} outlines the practical computational framework and presents numerical results. Section~\ref{sec:discussion} concludes the paper and discusses potential theoretical and computational extensions.

\section{Preliminaries}\label{sec:preliminary_definitions}

\subsection{The cryo-EM problem}
Cryo-EM reconstruction seeks to determine a 3-D molecular structure $\Phi$ from its 2-D noisy tomographic projections, taken at random viewing angles. 
In this work, we focus on the case of uniformly random rotations.  
Uniformity is often taken as a baseline model, was  the setting in Kam's paper \cite{kam1980reconstruction}, and is  \emph{harder} than the case of a non-uniform distribution of rotations in the sense that it requires asymptotically more images  (without sparsity priors) \cite{bandeira2017estimation,sharon2020method}. 
Formally, let $\mu$ be the Haar probability measure on the compact group $\SO$ of 3-D rotations, representing the uniform distribution. 
Assuming that we observe i.i.d.\ 2-D images of $\Phi$ after it has been randomly rotated according to $\mu$ and then tomographically projected to the plane, each projection image is modeled as:
\begin{equation} \label{eq:image-1}
    I_R(x,y) = \int_{z=-\infty}^{\infty} (R \cdot \Phi)(x,y,z) dz \, + \, \varepsilon(x,y), \quad  R \sim \mu,
\end{equation}
where $\varepsilon(x,y)$ is white  Gaussian noise with known variance $\sigma^2$, and $R \cdot \Phi$ denotes the action of the rotation $R$ on $\Phi$. 
Here, typically, the variance of the noise $\sigma^2$ is much greater than the magnitude of the clean projection. 

The cryo-EM problem is to estimate the molecular structure $\Phi$ from $n$ realizations of~\eqref{eq:image-1}, i.e., from the 2-D observations  $I_{R_1}, I_{R_2}, \ldots , I_{R_n}$.
In Section~\ref{sec:discussion}, we discuss how the proposed framework can be extended to account for additional aspects in the generative model for cryo-EM images.

\subsection{The method of moments}
Our theoretical and computational contributions are based on the method of moments---a basic statistical inference technique tracing back to the seminal paper of Karl Pearson in the end of the 19th century.   
Specifically, we use the second moment of the observations \eqref{eq:image-1}, and relate it to the sought-for 3-D structure. 

The (debiased) second observable moment is given by 
\begin{align}
    &\overline{M_2}((x_1, y_1), (x_2, y_2)) = \frac{1}{n} \sum_{i=1}^n I_{R_i}(x_1, y_1)  I_{R_i}(x_2, y_2)-B(\sigma^2), \label{eq:moments-empirical-second}
    \end{align}
    where $B(\sigma^2)$ is a bias term that depends only on the noise variance.
    For large enough $n$, we have
\begin{align}\label{eq:moments-population-second}
    \overline{M_2}((x_1, y_1), (x_2, y_2))  &\approx M_2((x_1,y_1),(x_2,y_2)),
\end{align}
where 
\begin{equation} \label{eq:pop-second-moment}
   M_2((x_1,y_1),(x_2,y_2)) := \!\!\! \int_{\SO} \!\!\!\!\! \!\!\!\!\! \!\! I_R(x_1,y_1) I_R(x_2,y_2)  d\mu(R)
   - B(\sigma^2),
\end{equation}
denotes the (debiased) population second moment, which is a function of $\Phi$ through Eq.~\eqref{eq:image-1}.
More precisely, for $n=\omega(\sigma^4)$ it holds $\overline{M_2} = M_2 + o(1)$ with high probability.

The idea of the method of moments is to find a structure $\Phi$ which matches 
the observable  moments.
It is an alternative to other standard statistical estimation methods, e.g.,  maximum likelihood estimation (MLE).  
That said, a recent paper suggests that in the low SNR regime, the method of moments approximates the  MLE~\cite{katsevich2020likelihood}. 

\subsection{Kam's method}
\label{sec:kam}
We detail a specific approach to autocorrelation analysis in cryo-EM, 
 introduced by Kam~\cite{kam1980reconstruction}. 
To this end, we need to introduce a convenient basis for representing a 3-D structure $\Phi$, the spherical Bessel basis~\cite{andrews1998special}. An expansion of maximum degree $L$ is defined by first expanding the Fourier transform $\mathcal{F}(\Phi)$ of $\Phi$ in spherical harmonics as 
\begin{equation}\label{eq:sph_fourier}
\mathcal{F}(\Phi) (k, \theta, \varphi)  \approx\sum_{\ell = 0}^L \sum_{m = -\ell}^\ell A_{\ell m}(k) Y^m_\ell(\theta, \varphi),
\end{equation}
where $k$ denotes the radial frequency and $Y^m_\ell(\theta, \varphi)$ are the spherical harmonics. 
In addition, the spherical harmonics coefficients $A_{\ell m}(k)$ are expanded by spherical Bessel functions, up to degree $S_\ell$, as
\begin{equation}
    A_{\ell m}(k) \approx  \sum_{s=1}^{S_\ell} a_{\ell m s}j_{\ell s}(k).
\end{equation}
The functions $j_{\ell s}(k)$ are the normalized spherical Bessel functions. 
By allowing $L$ and $S_\ell$ to grow unboundedly, any bandlimited function with bandlimit 1 can be represented in this basis.
However, when expanding 3-D molecular structures from discretized projection images, the Nyquist criterion applied to the projection images limits the amount of extractable information. This determines bounds on the maximally allowable truncation parameters $L$ and $S_\ell$; see the appendix for a detailed description.

We  aim to recover the coefficients $a_{\ell m s}$ (up to rotation and reflection in $\mathbb{R}^3$). As will be shown next, it is convenient to gather the coefficients into matrices $A_\ell$,  of size $S_\ell \times (2\ell+1)$, via $A_\ell(s,m) := a_{\ell m s}$, for each $\ell = 0, \ldots , L$.

Provided that the distribution of  viewing angles is uniform, Kam~\cite{kam1980reconstruction} showed that the second moment  of the Fourier transform of the projection images yields estimates for the following $S_\ell \times S_\ell$ matrices:
\begin{equation}\label{eq:defC}
C_{\ell}(s_1,s_2) := \sum_{m=-\ell}^{\ell} A_{\ell}(s_1, m)\overline{A_{\ell} (s_2,m)} = A_\ell A_\ell^*.
\end{equation}
Applying the Cholesky decomposition to each $C_\ell$ in \eqref{eq:defC} and imposing $\Phi$ to be real-valued, knowledge of \eqref{eq:defC} identifies each matrix of coefficient $A_{\ell}$ up to an unknown real, orthogonal transformation, provided $S_\ell \geq 2\ell +1$. That is,  we can compute $A_\ell O_\ell$ for some unknown orthogonal matrix $O_\ell$ in the group $O(2\ell+1)$.
Therefore, the second moment determines $\Phi$ up to a set of orthogonal matrices. To recover these matrices, and thus the 3-D structure, additional information is required.  In this paper, we suggest using a sparsity assumption.

\subsection{Crystallographic phase retrieval}
One of the contributions of this paper is to relate the cryo-EM reconstruction problem to crystallographic phase retrieval. Phase retrieval is the main computational challenge in X-ray crystallography, which is 
still a leading method for elucidating the atomic
structure of molecules.
The prevalence of crystallography is witnessed by the remarkable fact that  25 Nobel Prizes have been awarded for work directly or indirectly involving crystallography~\cite{galli2014x}.
Although there exist additional important  phase retrieval applications (see for example~\cite{shechtman2015phase,donatelli2015iterative,bendory2017fourier,barnett2022geometry}), 
X-ray crystallography is by far the most widely investigated  application.

The crystallographic phase retrieval problem entails recovering a sparse signal $x$ from its periodic autocorrelation (or, equivalently, from its Fourier transform magnitudes, namely, its power spectrum).
While simply stated, and despite its importance, the theoretical foundations of this problem continue to evolve. In particular, it was recently conjectured that a generic sparse signal can be recovered from its periodic  autocorrelation if the number of non-zero entries is smaller than half the signal's length~\cite{bendory2020toward}.
 This conjecture was verified for a few cases.
The relation of the crystallographic phase retrieval problem with  the beltway problem from combinatorial optimization is explored in~\cite{ghosh2021multi}. Our theoretical reconstruction guarantees in the following section can be viewed as analogous results in the setting of cryo-EM.

The standard algorithms for crystallographic phase retrieval build on two projection operators: one onto the measured data (the power spectrum) and the second onto the space of sparse signals. 
While simple algorithms that alternate between these two projections tend to quickly stagnate, a more sophisticated family of algorithms, based on reflections, shows excellent performance, though their running time is exponential in the sparsity level~\cite{elser2018benchmark,bendory2022sparse}.  
These algorithms are tightly related to splitting methods, such as  Douglas-Rachford and the alternating direction method of multipliers  (ADMM), and have been applied to a wide variety of problems~\cite{elser2007searching}.
A main contribution of this paper is a  modification of these algorithms to autocorrelation analysis for cryo-EM. In particular, we focus  on one such algorithm, called relaxed-reflect-reflect (RRR), but alternative algorithms, such as  Fienup’s hybrid input-output algorithm~\cite{fienup1982phase}, the difference map algorithm~\cite{elser2003phase}, and the relaxed averaged alternating reflections algorithm~\cite{luke2004relaxed}, can be adapted to cryo-EM by the same strategy. Importantly, if the model is accurate (e.g., no noise and the correct sparsity level is known) RRR iterations halt only when they find a solution that satisfies both constraints (defined by the projection operators). Thus, RRR does not suffer from local minima as gradient-based algorithms do.

\section{Superior sample complexity: The second moment suffices for sums of  point masses}\label{sec:theoretical}

This section presents our main theoretical result: the second moment suffices to recover an idealized sparse volume, i.e., a volume given as a weighted sum of point masses. 
We deduce that the second moment also suffices for a pixelated and blurred variant of the model.  
Our theorems imply an associated sample complexity of $n=\omega(\sigma^4)$. This stands in contrast to previous results which do not assume sparsity.  There, the third moment is required for recovery, and the associated sample complexity is $n=\omega(\sigma^6)$~\cite{bandeira2017estimation, fan2021maximum}.

\subsection{Models and main theoretical results} 
We use an atomistic representation of a molecule.  
In our first  idealized model, an atom is specified by a weighted Dirac delta function, and a molecule is a sum of such point masses.
In more detail, let $\ba_1, \ldots, \ba_p \in \R^3$ be the 3-D points representing atom locations, and  $w_1, \ldots, w_p$ be positive weights corresponding to the scattering potentials of the individual atoms.
Then,
\begin{equation}\label{eq:model_sparse_structure}
\Phi := \sum_{i=1}^p w_i \delta_{\ba_i},
\end{equation}
is the molecule composed of the atoms $(\ba_i, w_i)$.
Relabeling if necessary,  we assume that the $\ell^2$-norms $\|\ba_i\|$ are in descending order.

\begin{figure} 
\centering
\includegraphics[width = \linewidth]{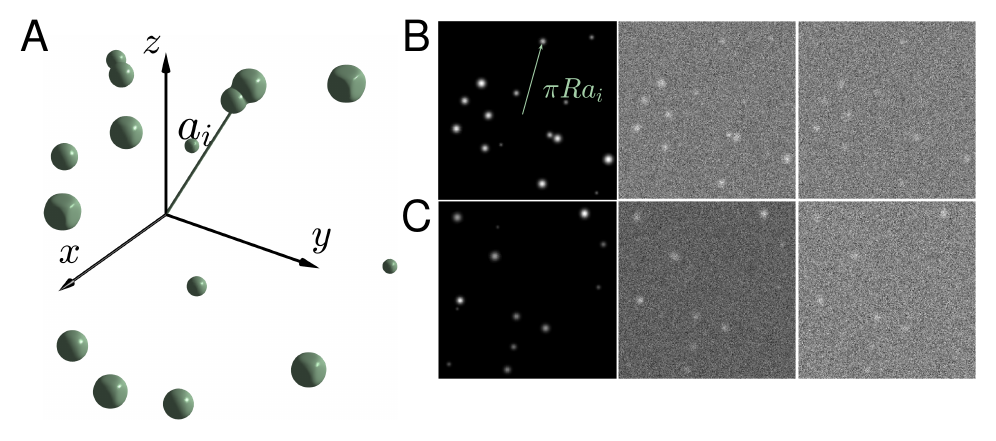}
\caption{(A) Example of a structure of the form in~\eqref{eq:model_sparse_structure}. For this illustration, the size of each dot is proportional to the weight $w_i$. (B--C) Two examples of tomographic projections of rotated versions of (A), with no noise (leftmost), signal-to-noise ratio 0.04 (middle) and signal-to-noise ratio 0.01 (rightmost).} \label{fig::definitions} 
\end{figure}
We model each projection image $I_R$ as a mixture of $p$ Dirac delta functions on $\R^2$ plus noise:
\begin{equation} \label{eq:image-2}
    I_R(x,y) := \sum_{i=1}^p w_i \delta_{\pi R \ba_i}(x,y)  +  \varepsilon(x,y).
\end{equation}
Here $\pi : \R^3 \rightarrow \R^2$ denotes coordinate projection onto the first two  coordinates, and $\varepsilon$ is white  Gaussian noise with (known) variance $\sigma^2$.
Given $n$ 2-D images as in \eqref{eq:image-2}, the reconstruction problem is to recover the atoms $(\ba_i, w_i)$ in \eqref{eq:model_sparse_structure} up to a global rotation and reflection. 
(The reflection ambiguity exists because a molecule and its reflection in the microscope's image plane are indistinguishable given cryo-EM data, see e.g. \cite{frank2006three}.)  
See Fig.~\ref{fig::definitions} for an illustration of the setup.

Under this model the (debiased) second population moment, obtained by substituting \eqref{eq:image-2} into 
\eqref{eq:pop-second-moment}, reads: 
\begin{align}
   \begin{split}
  &  M_2((x_1,y_1),(x_2,y_2)) = \sum_{i=1}^p \sum_{j=1}^p w_i w_j \, \times \\
  &\qquad\qquad\qquad\quad  \int_{\SO} \!\!\!\! \delta_{\pi R \ba_i}(x_1, y_1) \delta_{\pi R \ba_j}(x_2, y_2) d\mu(R), \label{eq:moment-population-second}
   \end{split}
\end{align}
where $\mu$ is the uniform distribution.
Note that $M_2$  is a measure on $\mathbb{R}^2 \! \times \mathbb{R}^2$.  

We introduce two assumptions on the atom locations:
\begin{enumerate}
    \item[\textbf{A1.}] The vectors $\ba_i$ are pairwise linearly independent;
    \item[\textbf{A2.}] The norms $\| \ba_i \|$ are distinct.
\end{enumerate}
We remark that conditions $\textbf{A1}$ and $\textbf{A2}$ are quite restrictive, e.g.,
ruling out molecules with nontrivial point-group symmetries.

Our first main theoretical result is stated as follows.

\begin{theorem} \label{thm:unique-determine}
Consider the model given by \eqref{eq:model_sparse_structure}-\textup{(\ref{eq:image-2})}.  
Assume that conditions \textup{\textbf{A1}-\textbf{A2}} hold.
Then, the support
of the second moment $M_2$  uniquely determines the set  of triples
    $\{ (\| \ba_i \|^2, \| \ba_j \|^2, \langle \ba_i, \ba_j \rangle ) : i, j= 1, \ldots, p \}$.
Therefore, $M_2$ \textup{(Eq.~\eqref{eq:moment-population-second})} uniquely determines the set of atom locations $\{ \ba_i : i = 1, \ldots, p \}$ up to a rotation and reflection in $\R^3$.
\end{theorem}
Theorem~\ref{thm:unique-determine} is proven in Subsection~\ref{sec:information}, after auxiliary results are given in Subsection~\ref{subsec:support}.  

Building on the uniqueness in Theorem~\ref{thm:unique-determine}, we obtain the following constructive result as well.

\begin{theorem} \label{thm:theoretical}
Consider the model given by \eqref{eq:model_sparse_structure}-\textup{(\ref{eq:image-2})}.  Assume that $p \geq 3$ and conditions  \textup{\textbf{A1}}-\textup{\textbf{A2}} hold. 
	Then, Algorithm~\ref{alg:theoretical} (described in Subsection~\ref{subsec:proof-recovery}) recovers the set of atoms 
	$\{ (\ba_i, w_i) : i = 1, \ldots, p \}$
	up to a rotation and reflection of the atom locations in $\R^3$ from the second moment $M_2$ \textup{(Eq.~\eqref{eq:moment-population-second})} in $\mathcal{O}(p^2)$ flops.  
\end{theorem}
Theorem~\ref{thm:theoretical} is proven in Subsection~\ref{subsec:proof-recovery}.
In fact the model of \eqref{eq:model_sparse_structure}-\textup{(\ref{eq:image-2})} can be extended to Gaussians of non-zero width of the form
\begin{equation}\label{eq:model_sparse_gaussian_structure}
\Phi(\bx) := \sum_{i=1}^p w_i k(\bx-\ba_i),
\end{equation}
where $k(\bx) = e^{-\frac{\|\bx\|^2}{2\kappa^2}}$ is an isotropic Gaussian with standard deviation $\kappa > 0$. Our results guarantee unique recovery from the second moment also for this model. The proof of the following result is given in the appendix.
\begin{theorem} \label{thm:unique-determine-gaussians}
Consider the model given by \eqref{eq:model_sparse_gaussian_structure}.  
Assume that conditions \textup{\textbf{A1}-\textbf{A2}} hold.
Then the second moment $M_2^G$ of the projection images formed from the model in \eqref{eq:model_sparse_gaussian_structure} uniquely determines the set of atom locations and weights $\{ (w_i,\ba_i) : i = 1, \ldots, p \}$ up to a joint rotation and reflection of the atom locations in $\R^3$.
\end{theorem}

	We next discuss an additional extension of the model of \eqref{eq:model_sparse_structure}-\textup{(\ref{eq:image-2})}. Prior works on the sample complexity of cryo-EM \textup{\cite{bandeira2017estimation, pereira2019information, abbe2018estimation}}  \textup{do not} directly apply to the models above.  
	The principal reason is that the measurement formation defined by \eqref{eq:image-2} is not finite-dimensional.
	Therefore, we consider a pixelated and blurred variant of the model.   
The molecule is still specified by a collection of atoms $\{(\ba_i, w_i): i = 1, \ldots, p\}$.
However, each projection image now consists of $2^m \! \times 2^m$ pixels:
\begin{multline} \label{eq:image-modified}
    I_R^{[m]}(j_1,j_2) :=  \! \int_{j_2 \tau }^{(j_2+1)\tau} \!\!\! \int_{j_1 \tau}^{(j_1 + 1) \tau} \!\!\! \left( \sum_{i=1}^{p} w_{i}  \delta_{\pi R \ba_i}(x,y)  \right) \\
     \ast  k(x,y) dx dy   +  \varepsilon(j_1, j_2).
\end{multline}
Here, we discretized $[-1,1]^2$ into equi-sized squares, where $j_1, j_2 \in \{-2^{m-1}, -2^{m-1}+1, \ldots, 2^{m-1} - 1\}$ and $\tau = 1/2^{m-1}$.   
Also, $\ast$ denotes convolution and $k(x,y)$ is the isotropic Gaussian kernel with fixed  variance $\kappa^2$, i.e., $k(x,y) = e^{\frac{-x^2-y^2}{2\kappa^2}}$.
Last, the noise satisfies
  $\varepsilon(j_1, j_2) \overset{i.i.d.}{\sim} \mathcal{N}(0,\sigma^2)$.

In the pixelated model, the (debiased) second population moment equals:
\begin{multline} \label{eq:secondmoment-modified}
    M_2^{[m]}((j_1, j_2), (j_3, j_4))  = \int_{j_4 \tau}^{(j_4 + 1) \tau} \! \int_{j_3 \tau}^{(j_3 + 1)  \tau} \!  \int_{j_2 \tau}^{(j_2 + 1) \tau} \! \int_{j_1 \tau}^{(j_1 + 1) \tau} \\ M_2((x_1,y_1), (x_2, y_2)) \! \ast \! \left(k(x_1, y_1) k(x_2, y_2)\right) dx_1 dy_1 dx_2 dy_2.
\end{multline}

We now state our main result for the pixelated model.  
Its proof relies on Theorems~\ref{thm:unique-determine} and~\ref{thm:theoretical}.

\begin{theorem} \label{thm:modified}

Consider the model given by \eqref{eq:image-modified}.
Fix an integer $p \geq 3$ and real numbers $r > 0$ and $w_{+} > w_{-} > 0$.
Assume that $\textup{\textbf{A1}-\textbf{A2}}$ hold, and for each $i=1, \ldots, p$ we have $ \| \ba_i \| \leq r$ and $w_{-} \leq w_i \leq w_{+}$.
Then, there exists   $m'=m'(p,r,w_+,w_-)$ with the following property.
Whenever $m \geq m'$ and $2^m \times 2^m$ pixels are used in \eqref{eq:image-modified}, 
 then the second moment $M_2^{[m]}$ \textup{(Eq.~\eqref{eq:secondmoment-modified})}
uniquely determines the set of atoms $\{ (\ba_i, w_i) : i = 1, \ldots, p \}$ up to a rotation and reflection in $\mathbb{R}^3$.
\end{theorem}

As the details are technical, we prove Theorem~\ref{thm:modified} in the appendix. 
We only use two properties of the Gaussian kernel $k$: that it is real-analytic and that its Fourier transform does not vanish.

\begin{corollary}
\label{cor:sample}
Assume the setting of Theorem~\ref{thm:modified} 
with $m \geq m'$. Then, the sample complexity for generic unique recovery (in the sense of \textup{\cite{bandeira2017estimation}}) is $n = \omega(\sigma^4)$ as $\sigma \rightarrow \infty$.
\end{corollary}

The rest of the section provides the proofs of Theorems~\ref{thm:unique-determine} and \ref{thm:theoretical}, with Theorem \ref{thm:modified}, Corollary \ref{cor:sample} and supporting results shown in the appendix.  
We emphasize that Algorithm~\ref{alg:theoretical} is a theoretical algorithm, not intended for use in practice due to its noise sensitivity as explained in Remark~\ref{rem:not-practical}. 
  By contrast,  Algorithm~\ref{alg:alternating} in the subsequent section is built for practical situations.

\subsection{Support of $M_2$} \label{subsec:support}

To recover the atoms from $M_2$, the main information that we use is actually qualitative.
Specifically, we rely on the particular structure of the support of the second moment $M_2$ in $\R^2 \times \R^2$.  To describe this, we need to first understand the possible images of one pair of atoms.

\begin{definition}\label{def:theta}
For $i, j = 1, \ldots, p$, let $\theta_{ij}: \SO \rightarrow \R^2 \times \R^2$ be the map given by $\theta_{ij}(R) = (\pi R \ba_i, \pi R \ba_j)$.
\end{definition}

\begin{definition}\label{def:Sij}
For $i, j = 1, \ldots, p$, let $S_{ij} \subseteq \R^2 \times \R^2$ be the image of $\theta_{ij}$, i.e., $S_{ij} = \{(\bx_1,\bx_2) \in \R^2 \times \R^2 : \exists R \in \SO \textup{ s.t. } \pi R \ba_i = \bx_1, \pi R \ba_j = \bx_2\}$.
\end{definition}

The next lemma characterizes $S_{ij} \subseteq \R^2 \times \R^2$ as the solution set to a system of polynomial equations and inequalities.  This will enable proof techniques from real algebraic geometry. 

\begin{lemma} \label{lem:Sij}
Assume $i \neq j$. Then the set $S_{ij}$ is connected, compact and semialgebraic. 
Letting $((x_1,y_1),(x_2,y_2))$ be variables on $\R^2 \times \R^2$, $S_{ij}$ is cut out by one quartic equation and two quadratic inequalities:
\begin{align} \label{eq:Sij}
   & \left( \| \ba_i \|^2 \!- x_1^2 \!- y_1^2 \right)\!\!\left( \| \ba_j \|^2 \! - x_2^2 \! - y_2^2 \right) \!=\! \left( \langle \ba_i, \ba_j \rangle \! - x_1 x_2 \! - y_1 y_2 \right)^2\!\!, \nonumber \\
  &  \, x_1^2 + y_1^2 \leq \| \ba_i \|^2 \quad \quad \quad \text{and}  \quad \quad \quad 
   x_2^2 + y_2^2 \leq \| \ba_j \|^2.
\end{align}
It has dimension $3$ as a semialgebraic set if condition $\textup{\textbf{A1}}$ holds.
\end{lemma}

A few different examples of the sets $S_{ij}$ are illustrated in Fig.~\ref{fig:Sij}, for varying values of $\langle \ba_i, \ba_j \rangle$ and $\|\ba_i\|, \|\ba_j\|$.
There we show the projection of $S_{ij}$ to $\mathbb{R}^3$ when $y_2$ is dropped.
 
\begin{figure}
    \centering
    \includegraphics[width = \linewidth]{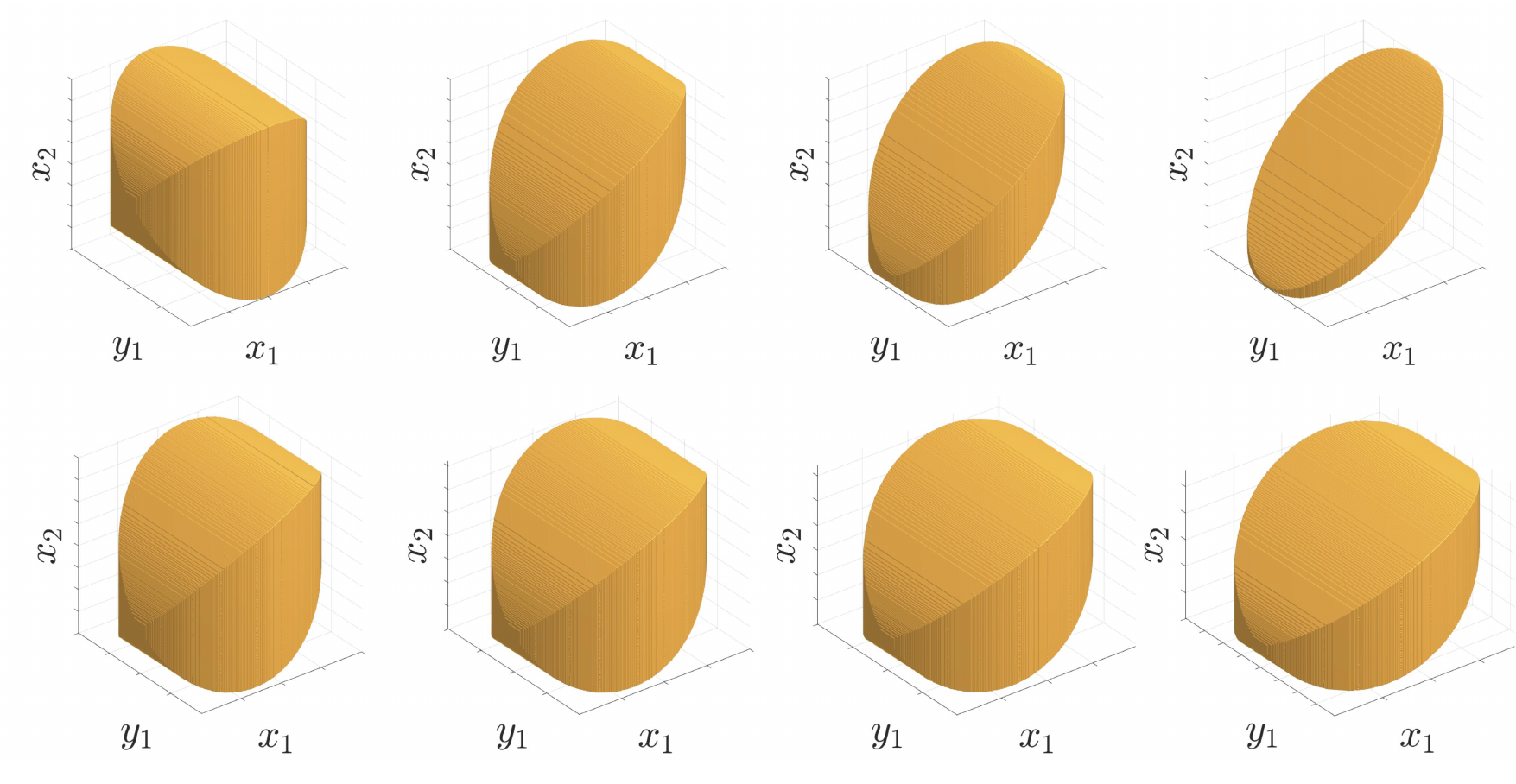}
    \caption{Illustration of the semialgebraic sets $S_{ij}$ described in Lemma~\ref{lem:Sij}, projected onto $(x_1, y_1, x_2)$-space. Upper row from left to right: $\|\ba_i\| = \|\ba_j\| = 1$ and $\langle \ba_i, \ba_j \rangle = \cos{\frac{\pi}{2}},\cos{\frac{\pi}{4}},\cos{\frac{\pi}{6}},\cos{\frac{\pi}{20}}$, respectively. Lower row from left to right: $\langle \ba_i , \ba_j \rangle = \|\ba_j\|\cos{\frac{\pi}{3}}$, $\|\ba_i\| = 1$ and $\|\ba_j\| = 1, 0.9, 0.8, 0.7$, respectively.  By Lemma~\ref{lem:support-M}, these sets make up the support of $M_2$, which is the basis of our theory.
    }
    \label{fig:Sij}
\end{figure}

The next result is immediate from Definitions~\ref{def:theta} and \ref{def:Sij}.

\begin{lemma} \label{lem:support-M}
The second moment is
\begin{equation} \label{eq:M2-pushforward}
  M_{2} = \sum_{i,j=1}^k w_i w_j (\theta_{ij})_{*}(\mu), 
\end{equation}
 where the subscripts indicate the pushforward measure defined by $(\theta_{ij})_{*}(\mu)(\cdot) = \mu \left( \theta_{ij}^{-1}(\cdot) \right)$.
In particular,  the support of 
$M_2$ 
is $\cup_{i,j=1}^p S_{ij}$.
\end{lemma}

\subsection{Information-theoretic uniqueness: Proof of Theorem~\ref{thm:unique-determine}} 
\label{sec:information}
We begin by proving Theorem~\ref{thm:unique-determine}. 
The key is a converse to Lemma~\ref{lem:Sij}.  
While Lemma~\ref{lem:Sij} implies the quartic equation in (\ref{eq:Sij}) (plus the quadratic inequalities there) determine the set $S_{ij}$, we need that $S_{ij}$ determines the quartic.
The proof of this converse uses results from real algebraic geometry~\cite{bochnak2013real}.

\begin{lemma}\label{lem:Zar-closure}
Assume that condition $\textup{\textbf{A1}}$ holds.  
Let $i \neq j$.  Then, the ideal of the real Zariski closure of $S_{ij}$ in $\R^2 \times \R^2$ is principal and generated by the quartic polynomial $q_{ij}$:
\begin{align} \label{eq:qij-expand}
    & \left( \| \ba_i \|^2 \!- x_1^2 \!- y_1^2 \right)\!\!\left( \| \ba_j \|^2 \! - x_2^2 \! - y_2^2 \right) - \left( \langle \ba_i, \ba_j \rangle \! - x_1 x_2 \! - y_1 y_2 \right)^2\!\! \nonumber \\[0.5em] 
   & = \left( \| \ba_i \|^2 \| \ba_j \|^2 - \langle \ba_i, \ba_j \rangle^2 \right) - \| \ba_j \|^2 x_1^2 - \| \ba_j \|^2 y_1^2 \nonumber \\[0.1em] 
   & \hspace{0.8em} - \| \ba_i \|^2 x_2^2 - \| \ba_i \|^2 y_2^2 
    + 2 \langle \ba_i, \ba_j \rangle x_1 x_2  + 2 \langle \ba_i, \ba_j \rangle y_1 y_2  \nonumber \\[0.1em]
    & \hspace{0.8em} + x_1^2 y_2^2 + y_1^2 x_2^2  - 2 x_1 y_1 x_2 y_2. 
\end{align}
Further, $q_{ij}$ is irreducible over $\R$.
\end{lemma}

\begin{corollary} \label{cor:irredundant}
Assume that conditions \textup{\textbf{A1}}-\textup{\textbf{A2}} hold.  
Then, the irredundant irreducible decomposition of the Zariski closure of the support of $M_2$ is
\begin{equation} \label{eq:irredundant}
    \{ (\bx_1, \bx_2)  : \bx_1 = \bx_2 \} \,\, \cup  \,\, \bigcup_{i \neq j} \overline{S_{ij}}.
  \end{equation}
\end{corollary}

We can now prove the information-theoretic uniqueness.

\begin{proof}[Proof of Theorem~\ref{thm:unique-determine}] 
The support of $M_2$ determines the real radical prime ideal of each of  top-dimensional irreducible component of its Zariski closure. 
By \textbf{A1}-\textbf{A2}, Corollary~\ref{cor:irredundant} and Lemma~\ref{lem:Zar-closure}, these ideals are $\langle q_{ij} \rangle = \{q_{ij}h \! : h \in \R[x_1, y_1, x_2, y_2]\}$ for $i \neq j$.  
The ideal $\langle q_{ij} \rangle$ uniquely determines $q_{ij}$, since the coefficient of $x_1^2y_2^2$ in \eqref{eq:qij-expand} is $1$. 
Extracting the coefficients of $x_2^2, x_1^2, x_1x_2$ in (\ref{eq:qij-expand}), $q_{ij}$ determines the triple $(\| \ba_i \|^2, \| \ba_j \|^2, \langle \ba_i, \ba_j \rangle)$.  
Ranging over $i,j$, we have proven that the support of $M_2$ fixes the set:
\begin{equation} \label{eq:triples}
    \{(\| \ba_i \|^2, \| \ba_j \|^2, \langle \ba_i, \ba_j \rangle ) : i \neq j  \}.
\end{equation}

By \textbf{A2}  and our assumption that the norms $\| \ba_i \|$ are descending,  knowledge of \eqref{eq:triples} lets us fill in the Gram matrix: 
\begin{equation} \label{eq:cholesky}
    G \, = \, A^{\top} A \, \in \, \R^{p \times p},
\end{equation}
where
\begin{equation} \label{eq:def-A}
    A \, = \, \begin{pmatrix}
    \ba_1 & \ldots & \ba_p 
    \end{pmatrix} \, \in \, \R^{3 \times p}.
\end{equation}
However $G$ determines $A$ up to left multiplication by a $3 \times 3$ orthogonal matrix. Indeed considering a truncated rank-3 eigendecomposition, we write
\begin{equation} \label{eq:truncated-eigen}
    G = Q D Q^{\top}\!,
\end{equation}
where $Q \in \R^{p \times 3}$ has orthonormal columns and $D \in \R^{3 \times 3}$ is diagonal and positive-semidefinite.  Then,
\begin{equation} \label{eq:truncated-eig}
    A = O D^{1/2} Q^{\top}\!,
\end{equation}
for some $O \in \operatorname{O}(3)$.
Therefore, the atoms' locations $\ba_i$ are determined up to a global rotation and reflection.
\end{proof}

\subsection{Recovery algorithm: Proof of Theorem~\ref{thm:theoretical}} \label{subsec:proof-recovery}

Now, we move forward and prove Theorem~\ref{thm:theoretical}.
We present Algorithm~\ref{alg:theoretical} for efficiently recovering the atoms $(\ba_i, w_i)$ from $M_2$.  
The algorithm is theoretical in that it relies on oracle access to the following information.
\begin{assumption}\label{assumption:oracle}
We assume oracle access to:
\begin{itemize}
    \item[\textup{\textbf{O1.}}] $\{ sample(S_{ij}) : i \neq j \}$, where $sample(S_{ij})$ consists of four or more Zariski-generic points on $S_{ij}$;
    \item[\textup{\textbf{O2.}}] 
    the value of the measure $M_2$ on the set $S_{ij}$ for all $i \neq j$. 
\end{itemize}
\end{assumption}

\begin{remark} \label{rem:not-practical}
In principle, $\textup{\textbf{O1}}$ and $\textup{\textbf{O2}}$ could be estimated from the sample moment $\overline{M_2}$~\textup{(\ref{eq:moments-empirical-second})} if $n=\omega(\sigma^4)$.
It would require the ability to identify points in the support of $M_2$ and cluster them according to the components $S_{ij}$.
However this encounters difficulty when dealing with noisy moments  discretized in pixels.  
The next section is dedicated to a different computational framework, which is suited for practical settings. 
\end{remark}

Proceeding, Algorithm~\ref{alg:theoretical} interpolates $sample({S_{ij}})$ to recover $q_{ij}$ from Lemma~\ref{lem:Zar-closure} (Eq.~\eqref{eq:qij-expand}).

\begin{lemma} \label{lem:interpolate}
Assume that condition \textup{\textbf{A1}} holds, and  \textup{\textbf{O1}} is known.
Let $sample({S_{ij}}) = \{ \left((x_{1k}, y_{1k}), (x_{2k}, y_{2k})\right) : k = 1, \ldots, |sample(S_{ij})| \}$.  Consider the matrix
\begin{equation} \label{eq:interpolation-matrix}
 \begin{pmatrix} 
& &  1  & & \\[0.2em] 
& &  x_{1k}^2 + y_{1k}^2  & & \\[0.2em] 
& \ldots &  x_{2k}^2 + y_{2k}^2  & \ldots & \\[0.2em] 
& &  x_{1k}x_{2k} + y_{1k} y_{2k}  &  & \\[0.2em] 
& &  x_{1k}^2 y_{2k}^2 + y_{1k}^2 x_{2k}^2 - 2x_{1k}y_{1k}x_{2k}y_{2k}  & & 
\end{pmatrix}^{\!\!\!\top}\!.
\end{equation}
Then it has rank $4$, with kernel spanned by
\begin{equation} \label{eq:right-kernel}
    \begin{pmatrix} \| \ba_i \|^2 \| \ba_j \|^2 - \langle \ba_i, \ba_j \rangle^2 \\[0.2em] 
    -\| \ba_j \|^2 \\[0.2em]  
    - \| \ba_i \|^2 \\[0.2em] 
    2 \langle \ba_i, \ba_j \rangle  \\[0.2em] 
    1 \end{pmatrix}.
\end{equation}
\end{lemma}

By Lemma~\ref{lem:interpolate}, we compute the triples $\{ (\| \ba_i \|^2, \| \ba_j \|^2, \langle \ba_i, \ba_j \rangle ) : i \neq j \}$ in $\mathcal{O}(p^2)$ time, by forming the matrices in Eq.~\eqref{eq:interpolation-matrix} and computing their kernels.
We then fill in the Gram matrix in Eq.~\eqref{eq:cholesky} as in the proof of Theorem~\ref{thm:unique-determine}. 
The atoms' locations $ \ba_i $ are recovered from the truncated eigendecomposition as in Eq.~\eqref{eq:truncated-eig}.  

The calculation of the weights $w_i$ is based on the following.

\begin{lemma} \label{lem:M2Sij}
Assume that conditions \textup{\textbf{A1}}-\textup{\textbf{A2}} hold. 
Then, for each $i \neq j$, the measure of $S_{ij}$ with respect to $M_2$ is
\begin{equation} \label{eq:ww}
    M_2(S_{ij}) =  w_{i} w_{j}.
\end{equation}
\end{lemma}

Therefore,  $\textbf{O2}$ tells us all off-diagonal entries of $ww^{\!\top} \! \in \R^{p \times p}$. 
We complete this uniquely to a rank-$1$ matrix by using
\begin{equation} \label{eq:w-complete}
(w w^{\!\top})_{ii} \, = \,  \frac{(w w^{\!\top})_{ij'} (w w^{\!\top})_{ji}}{(w w^{\!\top})_{jj'}},
\end{equation}
where $j, j'$ are any indices such that $i, j, j'$ are all distinct.
(This step requires $p \geq 3$.) 
The weights $w_i$ are lastly recovered either by computing the leading eigenvector/eigenvalue pair of $w w^{\top}$ or as the square root of the diagonal of $w w^{\top}$, using the fact that the $w$ are non-negative. 

\begin{remark}
We note that \eqref{eq:w-complete} is a particular case of the problem of recovering a low-rank matrix with corrupted diagonal entries; see, e.g., ~\textup{\cite{oseledets2006unifying,saunderson2012diagonal,saunderson2013diagonal}} for more on that problem. 
\end{remark}

We summarize the procedure of this section in Algorithm~\ref{alg:theoretical}.

\begin{algorithm}[tbh]
	\begin{algorithmic}
		\Require{Second population moment $M_2$ as in~\eqref{eq:moment-population-second}} 
		\Ensure{Atoms $\{ (\ba_i, w_i): i = 1, \ldots, p\}$ up to a rotation and reflection in $\R^3$}
		\begin{enumerate}
		\item Access \textbf{O1} in Assumption~\ref{assumption:oracle}
		\item Recover the unordered set $\{ (\| \ba_i\|^2, \| \ba_j \|^2, \langle \ba_i, \ba_j \rangle ) : i \neq j \}$ using \Cref{lem:interpolate}
		\item Fill in the Gram matrix $G = A^{\!\top}A$ with $A$ from \eqref{eq:def-A}
		\item Recover the $\ba_i$ up to orthogonal transformation by computing a truncated eigendecomposition of $G$ as in \eqref{eq:truncated-eig}
		\item Access \textbf{O2} in Assumption~\ref{assumption:oracle}
		\item Fill in the off-diagonal entries of $ww^{\!\top}$  using Lemma~\ref{lem:M2Sij}
		\item Complete $ww^{\!\top}$ using \eqref{eq:w-complete}
		\item Recover the $w_i$ from $ww^{\!\top}$ 
						\end{enumerate} 
		\Return $\{(\ba_i,w_i) : i = 1, \ldots, k\}$
		\caption{Recovering a sparse structure from its second moment}\label{alg:theoretical}
	\end{algorithmic}
\end{algorithm}

\begin{proof}[Proof of Theorem~\ref{thm:theoretical}]
The considerations above show that Algorithm~\ref{alg:theoretical} correctly recovers the set of atoms $\{(\ba_i, w_i) : i = 1, \ldots, p\}$  from $M_2$ (up to a  rotation and reflection in $\R^3$).
It costs $\mathcal{O}(p^2)$ in flops and storage once \textbf{O1} and \textbf{O2} are available if we use a randomized algorithm \cite{martinsson_tropp_2020} to compute the truncated decomposition in \eqref{eq:truncated-eigen}.
\end{proof}

\section{Kam's method with sparsity constraints} \label{sec:algorithm}

\begin{figure}[t]
\centering
\includegraphics[width = 0.82\textwidth]{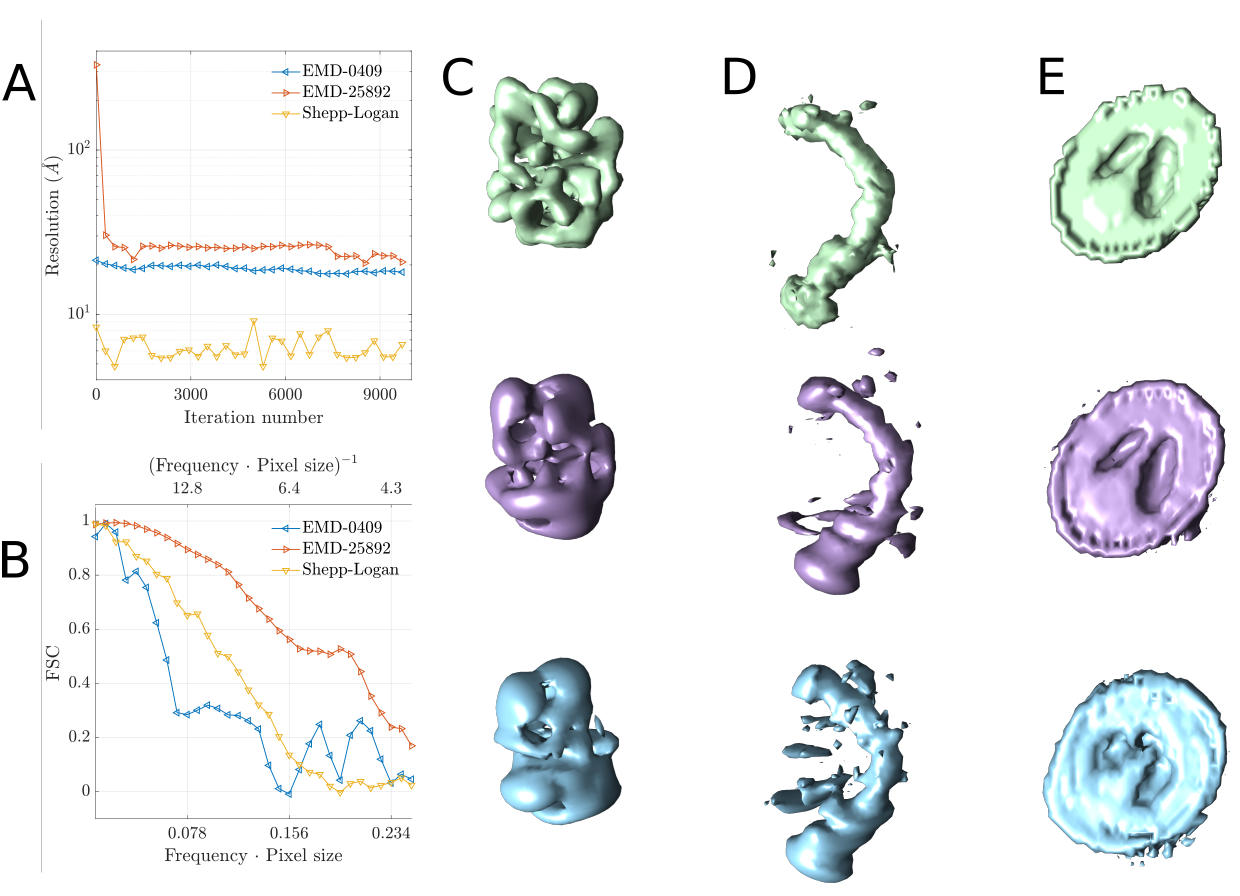}
\caption{Reconstruction results when applying the proposed algorithm to three example structures. (A) Resolution of the reconstructed volumes as a function of the number of iterations. The resolution is determined by the $0.5$ cutoff of the Fourier Shell Correlation (FSC). (B) Example FSC curves for reconstructed volumes as function of non-dimensionalized frequency. (C,D,E) Visualization of the reconstructed volumes for the three example structures: EMD-0409~\cite{herzik2019high}, EMD-25892~\cite{maker2022regulation} and the Shepp-Logan phantom~\cite{shepp1974fourier}, respectively. (Top) Ground-truth structure. (Middle) Truncation of ground-truth structure into the spherical Bessel basis using $L=8,12,12$, respectively. (Bottom) Reconstructed volume returned by Alg.~\ref{alg:alternating} using $L=8,12,12$, respectively. The visualizations were rendered by UCSF Chimera \cite{pettersen2004ucsf}.} \label{fig::ex_reconstruction}
\end{figure}

This section introduces  a computational framework to leverage sparsity in recovering the underlying molecular structure. The goal is to devise a principled way to compute ab initio approximations of the underlying structures, that can then be improved further in a refinement step (which is typically performed using expectation-maximization~\cite{scheres2012relion,punjani2017cryosparc} or used for model validation. In this section, we maintain the assumption that the distribution $\mu$ of viewing angles is uniformly distributed.
If $\mu$ is a non-uniform distribution, 
it is known that there is at most a finite list of structures that are consistent with the observed second order moment~\cite{sharon2020method}; employing sparsity to aid in the recovery problem with non-uniform distribution will be considered in future work.

We use projection-based optimization techniques from the related problem of crystallographic phase retrieval, coupled with information extracted from the second moment of the projection images. Without imposing the underlying sparsity, the second moment of the projection images determines the structure up to an ambiguity encoded by a set of unknown orthogonal matrices. The key idea of the algorithm is to  alternatingly project the molecular structure onto constraints encoded by the sparsity and by the projection image moments, respectively.

Analogously 
to~\eqref{eq:moment-population-second}, the moments of the projection images furnish information about the underlying 3-D structure. 
Unlike our theoretical results,  however, we consider a  general 3-D structure $\Phi$ expanded in a spherical Bessel basis as in Section~\ref{sec:kam}. As in the previous section, Gaussians (and mixtures of a small number of Gaussians) are often used as first approximation to the scattering potential of individual atoms. These Gaussians have the majority of their energy concentrated in a region of finite support. They can therefore be approximated by a small number of localized basis functions, such as the Haar wavelets. The sparse mixtures of Gaussians from the preceding sections can therefore be viewed as having a sparse representation in a wavelet basis, which offers computational advantages. This section therefore assumes that $\Phi$ can be represented by only a few wavelet coefficients. We also mention that recent work appearing after this paper was submitted shows that, for almost any basis, the second moment determines the structure uniquely, if the structure is sparse when expanded in the basis \cite{bendory2022sample}.
The next section introduces wavelet bases, and later we provide the details of the projection-based algorithm.

\subsection{Wavelet bases}
We  encode sparsity of a 3-D molecular structure $\Phi$ by a sparse expansion in wavelets~\cite{daubechies1992ten}---a popular choice of sparsifying, localized bases in a wide range of applications~\cite{mallat1999wavelet}. 
Our algorithm can easily be adapted
	to any specific wavelet basis, and, more generally,
	 to any choice of basis, for instance, sparsifying bases learned through data.

We denote the multilevel wavelet basis by $f_{m,n}$, where $m = 1, \ldots , m_{\text{max}}$ denotes the level of the wavelet and $n = 1, \ldots , n_{\text{max}}(m)$ the index of the function within the level.
As a shorthand, we define $W: \mathbb{R}^{M\times M \times M} \rightarrow \mathbb{R}^{m_{\text{max}}n_{\text{max}}(m)}$
as the map sending a 3-D structure to its vector of coefficients when expanded in the wavelet basis, i.e.,
\begin{equation}
W(\Phi) = \left( \langle \Phi, f_{m,n} \rangle \right)_{m,n = 1}^{m_{\text{max}}, n_{\text{max}}(m)}.
\end{equation}
Likewise, $W^{-1}: \mathbb{R}^{m_{\text{max}}n_{\text{max}}(m)} \rightarrow \mathbb{R}^{M\times M \times M}$ then maps a wavelet coefficient vector into its 3-D expansion by
\begin{equation}
W^{-1}\left( \left( c_{m,n} \right)_{m,n = 1}^{m_{\text{max}}, n_{\text{max}}(m)} \right) = \sum_{m,n} c_{m,n} f_{m,n}.
\end{equation}

An additional advantage of using wavelet bases is that $W$ and $W^{-1}$ can then be applied in linear time $\mathcal{O}(M^3)$ using a fast wavelet transform~\cite{mallat1989theory}.

\subsection{Projection-based algorithm}

For a discretized 3-D structure~$\Phi$ of size $M\times M \times M$, we define the mapping $\mathcal{SB}: \mathbb{R}^{M\times M \times M} \rightarrow \prod_{\ell=0}^L\mathbb{C}^{S_\ell \times (2\ell+1)}$ of the structure into its coordinates in the spherical Bessel basis by
\begin{align}
\mathcal{SB}(\Phi) = \left(A_0, A_1, \ldots , A_L\right).
\end{align}
The inverse mapping $\mathcal{SB}^{-1}:  \prod_{\ell=0}^L\mathbb{C}^{S_\ell \times (2\ell+1)} \rightarrow  \mathbb{R}^{M\times M \times M}$ then expands a set of coefficients in the spherical Bessel basis into its corresponding 3-D structure:    
\begin{align}
\mathcal{SB}^{-1} \left(A_0, \ldots , A_L\right) =  \mathcal{F}^{-1} \left( \sum_{\ell , m, s} \! a_{\ell m s}j_{\ell s}(k) Y^m_\ell(\theta, \varphi) \! \right).
\end{align}

As discussed in Section~\ref{sec:preliminary_definitions}, Kam's method identifies matrices $A_\ell O_\ell$, for $O_\ell$ an unknown orthogonal matrix, for each $\ell$ with $S_\ell \geq 2\ell + 1$. By possibly reducing the value of $L$ to the largest index with this property, we will for ease of notation assume that this property holds for $\ell = 0, \ldots , L$.
Therefore, at the onset of the algorithm, we have access to a set of  coefficient matrices $B = (B_0, \ldots , B_{L})$ satisfying
\begin{equation}\label{eq:def_ML}
B_\ell = A_\ell O_\ell, \quad O_\ell \in O(2\ell+1),
\end{equation}
for unknown orthogonal matrices $O_\ell$.  Our algorithm aims to recover an approximation of these unknown orthogonal matrices, which leads to an approximation of $\Phi$. This orthogonal matrix retrieval problem is an analogue to the problem of the missing phases in the phase retrieval problem~\cite{bendory2022algebraic}. We therefore adapt a popular algorithm from the phase retrieval literature into the problem of cryo-EM. The algorithm repeatedly utilizes two projections  onto the set of structures with a given sparsity level and a set determined by the projection images. These two projections are the main pillars of the algorithm and can be used in different ways, as explained next. 
But first, we  define the two projection operators. 

\subsubsection{First projection: Moment constraint.}\label{sec:second_proj}
We begin by defining  the first projection operator, denoted by $\rho_1$, as the projection onto the set defined by the $C_\ell$ matrices in \eqref{eq:defC}. Let $\mathcal{SB}(\Phi) = (A_0, \ldots , A_{L})$ in $\prod_{\ell=0}^{L}\mathbb{C}^{S_\ell \times (2\ell+1)}$ be the ordered collection of matrices of coefficients in the spherical Bessel basis. Define $\rho_1(\Phi)$ as the projection
\begin{align}\label{eq::second_proj}
\rho_1(\Phi) = \mathcal{SB}^{-1}\left(D_0, \ldots , D_{L}\right),
\end{align}
where the matrices $D_\ell$ are defined by 
\begin{align}\label{eq::procrustes}
\!\!(D_0,\ldots, D_{L}) \!=  \!\!\argmin_{(D_0,\ldots, D_{L})} \!\!\!\! \left\{ \| A_\ell \! -\! D_\ell\|_{\text F} : D_\ell D_\ell^* = C_\ell \right\},
\end{align}
with $C_\ell$ from \eqref{eq:defC}. \eqref{eq::procrustes} is an instance of the Orthogonal Procrustes problem. Although it is a non-convex optimization problem, it can be solved in closed form in terms of the singular value decomposition of $B_\ell^TA_\ell$, see e.g., \cite{schonemann1966generalized}. In the implementation, the matrices defining the operations $\mathcal{SB}$ and $\mathcal{SB}^{-1}$ are precomputed. The computational complexity of subsequently solving an instance of \eqref{eq::second_proj} is then $\mathcal{O}(L^4 + \sum_{l=0}^L S_\ell \ell^2 + M^3\log{M} + M^3\sum_{\ell}\ell S_\ell )  = \mathcal{O}(L^4 + M^3\log{M} + M^3L^3 )$, since typically $S_\ell = \mathcal{O}(L)$.

\subsubsection{Second projection: Sparsity constraint.}\label{sec:first_proj}
The projection $\rho_2$ promotes sparsity in a given local wavelet basis. For a structure~$\Phi$ and an integer~$K$, define $\rho_2(\Phi,K)$ as the structure with wavelet coefficients obtained by retaining the $K$ largest components of $W(\Phi)$ and replacing the remaining elements by zero, i.e.,
\begin{equation}\label{eq::first_proj}
    \rho_2(\Phi,K) = W^{-1}\left(\left( \alpha_{m,n}c_{m,n} \right)_{m,n = 1}^{m_{\text{max}}, n_{\text{max}}(m)}\right),
\end{equation}
where the coefficients are defined by $ W(\Phi) = \left( c_{m,n} \right)_{m,n = 1}^{m_{\text{max}}, n_{\text{max}}(m)}$ and $\alpha_{m,n} =1 $ if $c_{m,n}$ has magnitude among the $K$ largest magnitudes of the $c_{m,n}$, and zero otherwise.
The computational complexity of this step is $\mathcal{O}(M^3)$. We again emphasize that, generally, any localized basis or frame can be used to define $\rho_2$, and we fix a wavelet basis for the sake of definiteness.

\subsubsection{Algorithm.}
A straightforward algorithm to attempt to recover $\Phi$ is through alternating projections. This procedure is described by fixing a sparsity level $K$ and iterating the two projections $\rho_1$ and $\rho_2$ in turn. The use of the two projections in an alternating fashion is intended to promote convergence to an intersection point of the two sets. In the case of projecting onto convex sets, convergence results are known~\cite{cheney1959proximity}, but convergence is not guaranteed for the non-convex projections in \eqref{eq::second_proj} and \eqref{eq::first_proj}. Indeed, for non-convex sets, alternating projection schemes frequently suffer from convergence to local minima, and a method to escape the local minima is required. To achieve this, the phase retrieval literature details different iteration schemes combining the two projections $\rho_1$ and $\rho_2$ in different ways, for instance using the Relaxed-Reflect-Reflect (RRR) algorithm \cite{elser2017complexity,elser2018benchmark}. In terms of the projection operators, this iterative scheme can be written out as
\begin{align}
    \Phi^{(n+1/3)} &= \rho_1\left( \Phi^{(n)}\right), \nonumber\\
    \Phi^{(n+2/3)} &= \rho_2 \left(2  \Phi^{(n+1/3)} -  \Phi^{(n)},K \right), \\
    \Phi^{(n+1)} &=  \Phi^{(n)} + \beta \left( \Phi^{(n+2/3)} - \Phi^{(n+1/3)} \right), \nonumber
\end{align}
where $\beta \in (0,2)$ is a scalar hyperparameter. The  algorithm is summarized in Alg.~\ref{alg:alternating}.
As aforementioned, other phase retrieval algorithms which are based on two projection operators, such as the  difference map algorithm and the relaxed averaged alternating reflections algorithm, can be adapted to cryo-EM in the same fashion. 

\begin{algorithm}[tbh]
	\begin{algorithmic}
		\Require{Projection images, sparsity level $K$, maximum number of iterations $N$, hyperparameter $\beta$.}
		\Ensure{Estimated structure $\Phi$}
		\begin{enumerate}
		\item Form the matrices $C_\ell$ from \eqref{eq:defC}
		\item Compute the Cholesky factorizations of the $C_\ell$ to produce matrices $B_\ell$ in \eqref{eq:def_ML}
		\item $\Phi^{(0)} =  \mathcal{SB}^{-1}\left( B_0, \ldots , B_{L} \right) $
		\item For $n=0, \ldots , N-1$ do  
		\begin{itemize}
			\item  $\Phi^{(n+1/3)} = \rho_1\left( \Phi^{(n)}\right)$
		\item $\Phi^{(n+2/3)} = \rho_2 \left(2  \Phi^{(n+1/3)} -  \Phi^{(n)},K \right)$
		\item $\Phi^{(n+1)} =  \Phi^{(n)} + \beta  \left( \Phi^{(n+2/3)} - \Phi^{(n+1/3)} \right)$
		\end{itemize}
		
	   \end{enumerate}	
		\Return $\Phi^{(N-2/3)}$
		\caption{Recovering $\Phi$}\label{alg:alternating}
	\end{algorithmic}
\end{algorithm}

\begin{figure}[t]
\centering
\includegraphics[width = 0.8\textwidth]{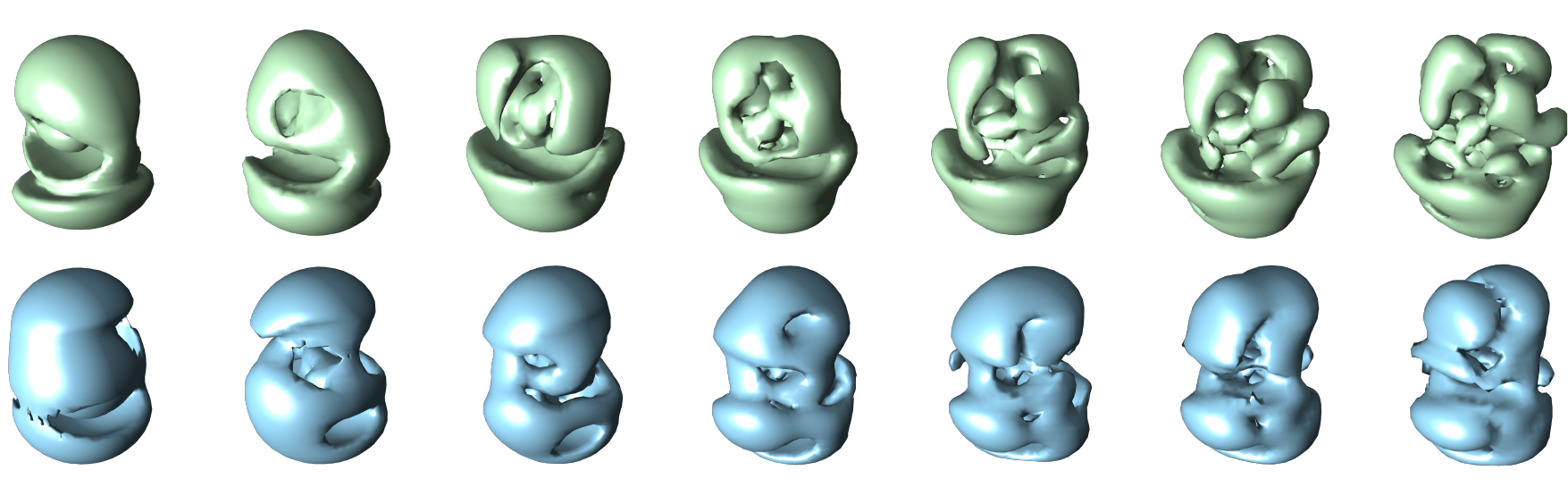}
\caption{(Bottom) Reconstruction results for EMD-0409 truncated with truncation parameter $L$ varying from $2$ (leftmost) to $8$ (rightmost). (Top) Ground truth structure truncated with truncation parameter $L$ varying from $2$ to $8$.}\label{fig:truncated_L}
\end{figure}

\subsection{Simulation results}
We apply Algorithm~\ref{alg:alternating} to two structures from the online EM data bank~\cite{lawson2016emdatabank}, EMD-0409~\cite{herzik2019high} and EMD-25892~\cite{maker2022regulation}, as well as the Shepp-Logan phantom~\cite{shepp1974fourier}. For each structure, we run Algorithm~\ref{alg:alternating} for a given value of $L$ and $K$ to obtain a reconstruction. 
To measure the reconstruction quality, we follow the standard procedure in the cryo-EM community,  and compute the Fourier shell correlation (FSC) between the estimated structure and the ground truth. Specifically, the FSC of two structures $\Phi_1$ and $\Phi_2$ is defined by
\begin{equation}
\text{FSC}(k) =  \frac{\sum_{r_i: \|r_i\| = k} \mathcal{F}\left(\Phi_1\right)(r_i) \overline{\mathcal{F}\left(\Phi_2\right)(r_i)}}{\sqrt{
\sum_{\|r_i\| = k} |\mathcal{F}\left(\Phi_1\right)(r_i)|^2
\sum_{\|r_i\| = k} |\mathcal{F}\left(\Phi_2\right)(r_i)|^2}}.
\end{equation}
where  one structure is the estimated structure, the second is the ground truth, and $\mathcal{F}$ denotes Fourier transform. The FSC is real-valued because of symmetry of the summation. The resolution is determined when the FSC curve drops below $0.5$.

EMD-0409 has dimensions $128\times 128 \times 128$, with each voxel having physical length of $1.117$ Å. EMD-25892 has dimensions $320\times 320 \times 320$, and voxel size $1.68$ Å. The volumes were downsampled by a factor of 2 and 5, respectively, to give structures of size $64\times 64 \times 64$. The ground truth matrices were generated exactly and the matrices $B_\ell$ in \eqref{eq:def_ML} were generated using $O_\ell$ chosen uniformly at random. To fix the units for the Shepp-Logan phantom, we assume the voxels to have side length 1 Å. The simulations used Haar wavelets to define $W$. The simulations set the values of the hyperparameters to $\beta = 0.5$ and $K=5000,4000,4000$, for EMD-0409, EMD-25892 and the Shepp-Logan phantom, respectively.

One iteration in Step 4 of Algorithm~\ref{alg:alternating} took around 4.6 seconds on a 2017 MacBook
Pro with a 3.1 GHz Intel Core i5 processor and 16 GB of memory. 10000 iterations therefore take around 13 hours.

The result of applying Algorithm~\ref{alg:alternating} to each structure is shown in Figure~\ref{fig::ex_reconstruction}. For all three example structures, during a run of the algorithm, the resolution initially rapidly improves. Afterwards, the improvement slows down and exhibits an exploratory and oscillating behavior.
This is typical for RRR-type algorithms, which frequently exhibit a rapid improvement in the early stages of the algorithm, followed by a long exploratory phase, where no improvement is made in the cost function, and then eventually followed by a final phase of rapid improvement. See Figure~5 in \cite{elser2018benchmark} for an example in the context of phase retrieval. The results of Algorithm~\ref{alg:alternating} exhibit the rapid initial improvement, but seem to not finish the exploration phase within the considered number of iterations. However, rather than expending more calculation time, Figure~\ref{fig::ex_reconstruction} shows that Algorithm~\ref{alg:alternating} obtains a reasonable ab initio model within roughly 1000 iterations, which can then be refined using other software packages like RELION or cryoSPARC \cite{sigworth1998maximum,scheres2012relion,punjani2017cryosparc}.
 
One could expect the obtained resolution to partly be limited by the sparsity constraint, since sparsity truncation will remove the finer details, i.e., high frequency information, although the algorithm compensates for this by projecting back onto the correct correlation. A few variations of Algorithm~2 could be considered, inspired by different alternatives to RRR in the phase retrieval literature (e.g.,~\cite{luke2004relaxed}), and one could alternatively allow for values of $K$ that increase with the iteration number in order to gradually increase the resolution.

The FSC curves in Figure~\ref{fig::ex_reconstruction} differ qualitatively from those obtained by other techniques, with values that typically approximately equal 1 at low resolutions and then fall off sharply at medium frequencies. This behavior is expected whenever the image rotations are accurately estimated. However, we operate at a much lower SNR for which rotations cannot be accurately assigned, leading to the different appearance of the FSC curves.

During the run of the algorithm, the optimal resolutions obtained for the three structures were $17.2$ Å, $20.52$ Å, and $4.59$ Å, respectively, and the resolutions at the initialization of the algorithm were $21.3$ Å, $329.8$ Å, and $8.4$ Å, respectively. As a comparison, the resolutions between the ground truth structures and their truncation into the spherical Bessel bases with the chosen values of $L$ are $6.0$ Å, $18.50$ Å and $2.20$ Å, respectively; these resolutions are bounds on the optimally obtainable resolutions. 

Figure~\ref{fig:truncated_L} also shows a comparison of EMD-0409 with its reconstruction, truncated to different values of $L$. Visually, the reconstructed element captures the relevant features of the ground truth although limited by the resolution expected from Figure~\ref{fig::ex_reconstruction}, for each value of $L$, and the resolution increases with $L$.

Movie~1 visualizes the reconstructed volume as a function of the iteration number. Note that the reconstruction at each step is visually similar to the ground truth, although the computed resolution noticeably improves during the run of the algorithm. This implies that even knowledge of the coefficients $A_\ell O_\ell$ with the wrong rotation matrices $O_\ell$ provides some information about the ground truth.

\begin{figure}[t]
\centering
\includegraphics[width = 0.8\textwidth]{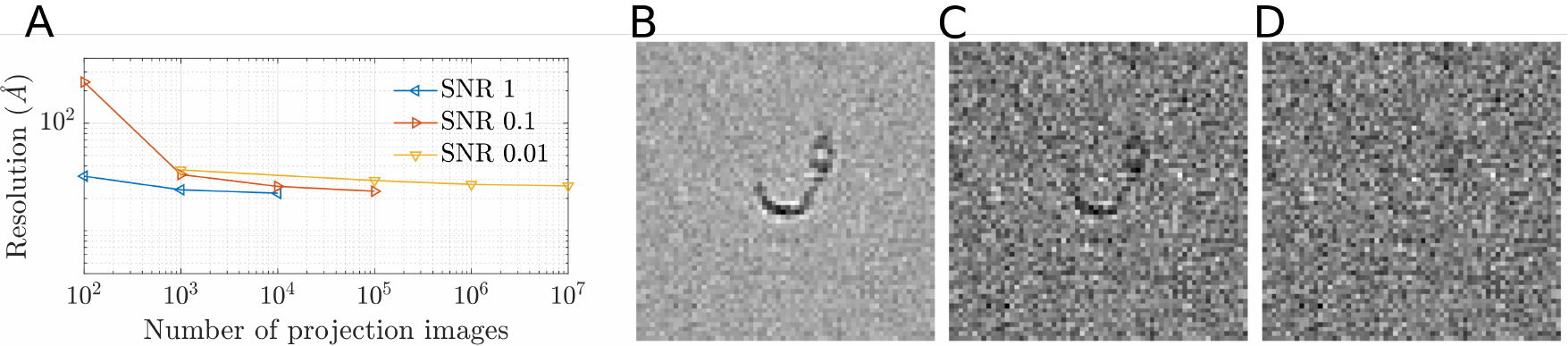}
\caption{(A) Resolution of reconstructed volume using second moment estimated from noisy projection images. (B--D) Sample CTF-affected and noisy projection images for SNR $1$, $0.1$ and $0.01$, respectively.}\label{fig:noisy_fig}
\end{figure}

We additionally show the result of running Algorithm~\ref{alg:alternating} on EMD-25892 with moments estimated from noisy projection images, which are also affected by contrast transfer functions (CTFs). We generate projection images according to the model
\begin{equation} \label{eq:image-ctf}
    I_{R_i}(x,y) = h_i(x,y) * \int_{z=-\infty}^{\infty} (R_i \cdot \Phi)(x,y,z) dz \, + \, \varepsilon(x,y),
\end{equation}
where $h_i$ is a point-spread function. The images were generated using signal-to-noise ratios 1,0.1 and 0.01, with the number of images used ranging between $10^2$ and $10^7$. All projection images used voltage $300$ kV and spherical aberration $2$ mm. The defocus values ranged between 1 $\mu$m and 4 $\mu$m.  When using at most $10^5$ projection images, the images used distinct defocus values. For $10^6$ and $10^7$ images, we divided the images into $10^4$ and $10^3$ defocus groups, respectively.

We estimate the second moment using the approach in \cite{marshall2022fast}. The result is shown in Fig.~\ref{fig:noisy_fig}, with accuracy comparable to that of Fig.~\ref{fig::ex_reconstruction} when using $\omega(\sigma^4)$ projection images. The simulations used $\beta = 0.5$ and $K=5000$.

\section{Discussion} \label{sec:discussion}
The contribution of this paper is twofold. As the first contribution, our theoretical results imply that a sparse mixture of  point masses can be uniquely recovered from the second order moment, even in the case of a uniform distribution of viewing angles, whereas previous work has only proven recovery using the third order moment. 
Thus, fewer images are required for reconstruction. This has a number of potential experimental implications. Firstly, since microscope time is expensive, this may greatly reduce the cost of the experimental part of the cryo-EM pipeline. 
This might be especially important for XFEL, where  throughput is a major bottleneck and viewing directions are more likely to be uniformly distributed~\cite{spence2017xfels,kirkwood2022multi}.
Secondly, it may enable reconstruction of structures where a limited number of projection images can be captured.
This might be the case, for example,  when the molecule may appear in several conformational states, and a limited number of images will be available for each  conformation.
 
The second contribution is a new algorithm for ab initio modeling which can be used as a starting point for iterative refinement procedures and additionally provides another way to validate reconstruction results obtained by different computational techniques. The computational framework introduced in this article  opens the door to incorporating a number of promising techniques from crystallographic phase retrieval into cryo-EM algorithms. There is, for instance, flexibility in choosing the projection operators $\rho_1$ and $\rho_2$.
It may include biologically-oriented priors, such as minimum atom-atom distance or Wilson statistics~\cite{singer2021wilson,gilles2022molecular}, or data-driven priors based on previously resolved structures~\cite{kimanius2021exploiting}.  A systematic study of adapting these techniques will be initiated in coming work. Additional future work includes extending the use of sparsifying priors in other parts of the cryo-EM reconstruction pipeline, for instance in existing approaches to iterative refinement~\cite{scheres2012relion} or in autocorrelation analysis using micrographs without particle picking~\cite{bendory2018toward,lan2020multi,bendory2021multi}. However, we do not expect sparsity to have as dramatic an impact on the sample complexity in the case of reconstruction directly from micrographs without particle picking. When expanding $L\times L$ projection images in a steerable basis such as the Fourier-Bessel basis \cite{marshall2022fast}, the second order moment of picked particles has $\mathcal{O}(L^3)$ independent entries. This is comparable to the number of parameters required to describe the 3-D structure. Still, in the case of uniform distribution of viewing directions Kam \cite{kam1980reconstruction} showed that the second order moment is insufficient for 3-D structure recovery, but this paper shows that additional sparsity assumptions ensure unique recovery of a 3-D structure from the second order moment. For autocorrelation analysis of entire micrographs, the second order moment is a 1-D profile equivalent to a rotationally-invariant power spectrum. Therefore, the number of entries is clearly insufficient for 3-D reconstruction and information from the third order moment needs to be incorporated as well. The sparsity constraint may potentially improve the quality of recovery, but the sample complexity is asymptotically the same, proportional to $\sigma^6$.
  
  Yet another important direction is to incorporate the sparsity prior into reconstruction by the method of moments when there is a non-uniform distribution of viewing directions \cite{sharon2020method}.

Molecular reconstruction using the method of moments fills an important niche in single-particle reconstruction. Existing software packages like RELION \cite{scheres2012relion}, cryoSPARC \cite{punjani2017cryosparc} etc. encounter difficulties when reconstructing small molecules (e.g., below 40 kDa) even though particle picking is not prohibitively difficult at this size \cite[Fig.~10f,g,h]{vinothkumar2016single}. The methods of this paper are therefore viable in situations where other techniques are not expected to be. However, we emphasize that we do not suggest the method of moments to be competitive to RELION or cryoSPARC in terms of resolution for large molecules except for the purpose of validation or fast ab-initio modeling technique. Moreover, we also expect that incorporating sparsity priors would improve the sample complexity and quality of reconstruction algorithms employed by existing software packages like RELION and cryoSPARC. A full demonstration would be an important direction for future work.

\appendix

\section{Proofs of auxiliary results for Theorems~\ref{thm:unique-determine} and~\ref{thm:theoretical}}
\subsection{Proof of Lemma~\ref{lem:Sij}}
Note that $S_{ij}$ is connected and compact, since $S_{ij} = \theta_{ij}(\SO)$ and $\SO$ is compact and connected, while $\theta_{ij}$ is continuous.
It also is semialgebraic, as $\SO$ is a real algebraic variety and $\theta_{ij}$ is a polynomial map (see the Tarski-Seidenberg theorem \cite{bochnak2013real}). 

Define $T_{ij} \subseteq \R^2 \times \R^2$ as the set cut out by the three constraints in \eqref{eq:Sij}.  
Assume $((x_1,y_1),(x_2,y_2)) \in S_{ij}$. 
By definition of $S_{ij}$, there exist $R \in \SO$ and $z_1, z_2 \in \R$ such that $R \ba_i = (x_1, y_1, z_1)^{\!\top}$ and $R \ba_j = (x_2, y_2, z_2)^{\!\top}$.  
Then
\begin{equation*}
    \| \ba_i \|^2 - x_1^2 - y_1^2 = \| R \ba_i \|^2 - x_1^2 - y_1^2 = z_1^2,
\end{equation*}
likewise 
$
\| \ba_j \|^2 - x_2^2 - y_2^2 = z_2^2,
$
and 
\begin{align*}
\left( \langle \ba_i, \ba_j \rangle \! - x_1 x_2 \! - y_1 y_2 \right)^2 &= \left( \langle R \ba_i, R \ba_j \rangle \! - x_1 x_2 \! - y_1 y_2 \right)^2 \\ &= (z_1z_2)^2.
\end{align*}
Also, $x_1^2 + y_1^2 \leq \| R \ba_i \|^2 = \| \ba_i \|^2$, and similarly $x_2^2 + y_2^2 \leq \| \ba_j^2 \|^2$.
These show that $\left( (x_1, y_1), (x_2, y_2) \right) \in T_{ij}$, whence $S_{ij} \subseteq T_{ij}$.

For the converse, take $\left( (x_1, y_1), (x_2, y_2) \right) \in T_{ij}$.  
Let 
\begin{equation*}
    z_1 = \sqrt{\| \ba_i \|^2 - x_1^2 - y_1^2} \quad\text{and}\quad z_2 = \varepsilon \sqrt{ \| \ba_j \|^2 - x_2^2 - y_2^2},
\end{equation*}
where 
$\varepsilon = \operatorname{sign}\left( \langle \ba_i, \ba_j \rangle - x_1 x_2 - y_1 y_2 \right)$.  
Put $\mathbf{b}_i = (x_1, y_1, z_1)^{\!\top}$ and $\mathbf{b}_j = (x_2, y_2, z_2)^{\!\top}$ in $\R^3$.  
By the choice of $z_1$ and $z_2$, 
\begin{equation} \label{eq:norms}
\| \mathbf{b}_i \| = \| \ba_i \| \quad \text{and} \quad \| \mathbf{b}_j \| = \| \ba_j \|.
\end{equation}
Also, from the equality constraint in \eqref{eq:Sij}, it holds
$z_1^2 z_2^2 = \left( \langle \ba_i, \ba_j \rangle - x_1x_2 - y_1y_2 \right)^2 $.  This with the choice of $\varepsilon$ implies 
\begin{equation} \label{eq:angle}
\langle \mathbf{b}_i, \mathbf{b}_j \rangle = \langle \ba_i, \ba_j \rangle.
\end{equation}
From \eqref{eq:norms} and (\ref{eq:angle}), there exists $R \in \SO$ such $\mathbf{b}_i = R \ba_i$ and $\mathbf{b}_j = R \ba_j$.  Hence $\left( (x_1, y_1), (x_2, y_2) \right) \in S_{ij}$, whence $T_{ij} \subseteq S_{ij}$.  
We conclude  $T_{ij} = S_{ij}$.

The dimension of $S_{ij}$ as a semialgebraic set is the maximal dimension of a cell in any cylindrical algebraic decomposition of it \cite[Cor.~2.8.9]{bochnak2013real}.
This agrees with the maximal rank attained by the differential of $\theta_{ij}$:  
\begin{equation*}
\dim(S_{ij}) = \max_R  \operatorname{rank} (D \theta_{ij} : T_R(\SO) \rightarrow T_{\theta_{ij}(R)}(\R^2 \times \R^2)),
\end{equation*}
where $T$ denotes tangent space.  We recall that the tangent space to rotation matrices is parameterized by skew-symmetric matrices.
Specifically,
$T_R(\SO) = \{  [s]_{\times}R : s \in \R^3\}$, where 
\begin{equation*}
[s]_{\times} := \begin{pmatrix} 0 & s_3 & -s_2 \\ -s_3 & 0 & s_1 \\ s_2 & -s_1 & 0 \end{pmatrix}.
\end{equation*}
Then, $D\theta_{ij}( [s]_{\times}R) = \left(\pi[s]_{\times} R \ba_i, \pi[s]_{\times} R \ba_j \right) \in \R^2 \times \R^2 = T_{\theta_{ij}(R)}(\R^2 \times \R^2)$. 
Putting $(x_1, y_1, z_1)^{\!\top} := R \ba_i$ and $(x_2, y_2, z_2)^{\top} := R \ba_j$, we rewrite
\begin{equation*} \label{eq:43-matrix}
  D \theta_{ij}([s]_{\times} R) = W(R) s, \quad \quad \text{where}~W(R) :=  \begin{pmatrix} 0 & -z_1 & y_1 \\
    z_1 & 0 & -x_1 \\
    0 & -z_2 & y_2 \\
    z_2 & 0 & -x_2
    \end{pmatrix}.
\end{equation*}
Thus, $\dim(S_{ij}) = 3$, unless $W(R)$ is rank-deficient for all $R \in \SO$.
We claim it is rank-deficient for specific $R$ if and only if $(x_1, y_1, z_1)$ and $(x_2, y_2, z_2)$ are linearly dependent or $z_1 = z_2 = 0$.  This is proven using a computer algebra system, e.g. \cite{M2}.  
Indeed, if $I$ is ideal in the ring $\mathbb{Q}[x_1,y_1,z_1,x_2,y_2,z_2]$ generated by the $3 \times 3$ minors of $W(R)$, the claim follows from calculating the primary decomposition \cite{eisenbud2013commutative}:
\begin{equation*}
    I = \langle z_1 y_2 - y_1 z_2, 
    z_1 x_2 - x_1 z_2,  
    y_1 x_2 - x_1 y_2
    \rangle \cap 
    \langle z_2^2, z_1 z_2, z_1^2 \rangle.
\end{equation*}
Given the claim, $\operatorname{rank}(W(R)) < 3$ for all $R \in \SO$ if and only if $\ba_i$ and $\ba_j$ are linearly dependent.
In other words, $S_{ij}$ has dimension $3$ if and only if $\ba_i$ and $\ba_j$ are linearly independent.  
The proof of
 Lemma~\ref{lem:Sij} is complete. \hfill $\qed$

\subsection{Proof of Lemma~\ref{lem:support-M}}
As mentioned in the main text, this is immediate from Definitions~\ref{def:theta} and \ref{def:Sij}  and~\eqref{eq:moment-population-second}. \hfill $\qed$

\subsection{Proof of Lemma~\ref{lem:Zar-closure}}

If $X$ is a subset of a real Euclidean space $\R^k$, we write $\overline{X}$ for the Zariski closure in $\R^k$, and  $\mathcal{I}(\overline{X})$ for the real radical ideal of $\overline{X}$.

First, note that $\overline{S_{ij}}$ is irreducible.
This is because $S_{ij} = \theta_{ij}(\SO)$, $\theta_{ij}$ is polynomial and $\SO$ is an irreducible algebraic variety.  
So, $\mathcal{I}(\overline{S_{ij}})$ is prime~\cite[Thm.~2.8.3(ii)]{bochnak2013real}.
Also, $\overline{S}_{ij}$ has dimension $3$ as an algebraic variety.  
This is by \cite[Prop.~2.8.2]{bochnak2013real}, and Lemma~\ref{lem:Sij} which states $S_{ij}$ has dimension $3$ as a semialgebraic set.   
So, $\mathcal{I}(\overline{S_{ij}})$ has height $1$  \cite[Def.~2.8.1]{bochnak2013real}.  
Here, every prime ideal with height $1$ is principal, as $\mathbb{R}[x_1, y_1, x_2, y_2]$ is a unique factorization domain.  
It follows that
\begin{equation} \label{eq:f}
\mathcal{I}(\overline{S_{ij}}) = \langle f \rangle,
\end{equation}
for some irreducible polynomial $f \in \mathbb{R}[x_1, y_1, x_2, y_2]$, where angle brackets indicate ideal generation.

 By Lemma~\ref{lem:Sij}, we know $S_{ij} \subseteq \mathcal{Z}(q_{ij})$, where $\mathcal{Z}$ denotes the zero set in $\R^2 \times \R^2$.
Taking closures,  $\overline{S_{ij}} \subseteq \mathcal{Z}(q_{ij})$. 
Equivalently, 
$q_{ij} \in \mathcal{I}(\overline{S_{ij}})$.
By \eqref{eq:f}, this means $f$ evenly divides $q_{ij}$, say, 
\begin{equation} \label{eq:factorization}
q_{ij} = fg,
\end{equation}
for some $g \in \R[x_1, y_1, x_2, y_2]$.
To conclude the proof, it suffices to prove that $g$ is a nonzero scalar.
Then $\langle q_{ij} \rangle = \langle f \rangle = \mathcal{I}(\overline{S_{ij}})$, and $q_{ij}$ is irreducible because $f$ is.

For a contradiction,  assume that $f$ has positive degree.  
Then, \eqref{eq:factorization} implies
\begin{equation} \label{eq:top}
    (q_{ij})_{\text{top}} = f_{\text{top}} g_{\text{top}},
\end{equation}
where the subscript indicates the top total degree part of the polynomial.  Here,
\begin{align} \label{eq:qij-expandSI}
 q_{ij} = & \left( \| \ba_i \|^2 \| \ba_j \|^2 - \langle \ba_i, \ba_j \rangle^2 \right) - \| \ba_j \|^2 x_1^2 - \| \ba_j \|^2 y_1^2  - \| \ba_i \|^2 x_2^2 - \| \ba_i \|^2 y_2^2 \nonumber \\[0.1em] & + 2 \langle \ba_i, \ba_j \rangle x_1 x_2  + 2 \langle \ba_i, \ba_j \rangle y_1 y_2  + x_1^2 y_2^2 + y_1^2 x_2^2  - 2 x_1 y_1 x_2 y_2.
\end{align}
Thus,
\begin{equation*}
    (q_{ij})_{\text{top}} = x_1^2y_2^2 + y_1^2 x_2^2 - 2 x_1 y_1 x_2 y_2 = (x_1 y_2 - y_1 x_2)^2.
\end{equation*}
From \eqref{eq:top}, the assumption that $f$ has positive degree and unique factorization, we deduce that (possibly after multiplying by nonzero scalars) 
\begin{equation*}
f_{\text{top}} = g_{\text{top}} = x_1 y_2 - y_1 x_2.
\end{equation*}
Therefore,
\begin{align} \label{eq:gf-expand}
    f = x_1 y_2 - x_2 y_1 + \alpha x_1 + \beta y_1 + \gamma x_2 + \delta y_2 + \varepsilon, \nonumber \\
    g = x_1 y_2 - x_2 y_1 + \zeta x_1 + \eta y_1 + \theta x_2 + \iota y_2 + \kappa,
\end{align}
for some $\alpha, \beta, \gamma, \delta, \varepsilon, \zeta, \eta, \theta, \iota, \kappa \in \R$.
Now we insert \eqref{eq:gf-expand} and \eqref{eq:qij-expandSI} into \eqref{eq:factorization}.
Equating the constants and the coefficients of $x_1^2, y_1^2, x_2^2, y_2^2$ gives
\begin{equation} \label{eq:first-equate}
    \begin{cases}
    \| \ba_i \|^2 \| \ba_j \|^2 - \langle \ba_i, \ba_j \rangle^2 = \varepsilon \kappa \\
    - \| \ba_j \|^2 = \alpha \zeta \\
    - \| \ba_j \|^2 = \beta \eta \\
    - \| \ba_i \|^2 = \gamma \theta \\
  - \| \ba_i \|^2 = \delta \iota.
    \end{cases}
\end{equation}
All the left-hand sides in \eqref{eq:first-equate} are nonzero because $\ba_i$ and $\ba_j$ are linearly independent.  
Therefore, $\alpha, \ldots, \kappa$ are all nonzero.
Next, we equate the coefficients of $x_1, y_1, x_2, y_2$ in \eqref{eq:factorization}.  The result is that
\begin{equation*}
    \begin{pmatrix} 
    \alpha & \zeta \\
    \beta & \eta \\ 
    \gamma & \theta \\
    \delta & \iota 
    \end{pmatrix} \!\! \begin{pmatrix} 
    \kappa \\
    \varepsilon
    \end{pmatrix} = 0, 
\end{equation*}
whence the $4 \times 2$ matrix has rank $1$.  So all its $2 \times 2$ minors vanish, in particular
\begin{equation} \label{eq:diff}
    \alpha \eta - \beta \zeta = 0.
\end{equation}
Finally, we equate the coefficients of $x_1y_1$ in \eqref{eq:factorization}:  
\begin{equation} \label{eq:sum}
    0 = \alpha \eta + \beta \zeta. 
\end{equation}
\eqref{eq:diff} and (\ref{eq:sum}) imply
\begin{equation}
    \alpha \eta = \beta \zeta = 0.
\end{equation}
But this contradicts the earlier finding that $\alpha, \ldots, \kappa$ are all nonzero.  So the assumption that $f$ has positive degree is false, and $f$ is a nonzero scalar.
This proves Lemma~\ref{lem:Zar-closure}. \hfill $\qed$

\subsection{Proof of Corollary \ref{cor:irredundant}}
  By Lemma~\ref{lem:support-M}, the support of $M_2$ is 
  $\cup_{i,j=1}^p S_{ij}$.  This has  Zariski closure $\cup_{i,j=1}^p \overline{S_{ij}}$.
  We claim its irredundant irreducible decomposition is
  \begin{equation} \label{eq:irredundantSI}
    \{ (\bx_1, \bx_2)  : \bx_1 = \bx_2 \} \,\, \cup  \,\, \bigcup_{i \neq j} \overline{S_{ij}}.
  \end{equation}

The claim follows from several facts.
First, for all $i \neq j$, $\overline{S_{ij}}$ is irreducible.  
 It has dimension $3$ and defining equation $q_{ij}$ (\ref{eq:qij-expandSI}), by \textbf{A1} and Lemma~\ref{lem:Zar-closure}.
When $i \neq j$, $i' \neq j'$ and
$(i,j) \neq (i',j')$, 
then $q_{ij}$ and $q_{i'j'}$ are not scalar multiples of each other by \textbf{A2}
(cf. their coefficients on $x_1^2, x_2^2, x_1^2y_1^2$).  
Hence $\overline{S_{ij}} \neq \overline{S_{i'j'}}$.
Next, $\overline{S_{ii}} = \{ (\bx_1, \bx_2) : \bx_1 = \bx_2 \}$ (all $i$).  Further, $\{ (\bx_1, \bx_2) : \bx_1 = \bx_2 \} \nsubseteq \overline{S_{ij}}$ (all $i \neq j$), because when we substitute $x_1=x_2$ and $y_1 = y_2$ into \eqref{eq:qij-expandSI} we get a nonzero result as the constant term does not vanish.
All together, (\ref{eq:irredundantSI}) is  the claimed irredundant irreducible decomposition as wanted. \hfill $\qed$

\subsection{Proof of Lemma~\ref{lem:interpolate}}
This follows from $\mathcal{I}(\overline{S_{ij}}) = \langle q_{ij} \rangle$ (Lemma~\ref{lem:Zar-closure}), the expression (\ref{eq:qij-expandSI}) for $q_{ij}$, and the proof of \cite[Thm.~3]{chen2019numerical}. \hfill $\qed$ 

\subsection{Proof of Lemma~\ref{lem:M2Sij}}
From \eqref{eq:M2-pushforward}, we have
\begin{align*}
    M_2(S_{ij}) = \sum_{i',j'=1}^p w_{i'} w_{j'} \mu(\theta_{i'j'}^{-1}(S_{ij})).
\end{align*}
Here $\mu(\theta_{ij}^{-1}(S_{ij})) = \mu(\SO) = 1$ by definition of $S_{ij}$. 
On the other hand, for all $i', j' = 1, \ldots, p$ with $(i', j') \neq (i, j)$,  $S_{ij} \cap S_{i'j'}$ is a semialgebraic set with positive codimension in $S_{i'j'}$ by the fact that (\ref{eq:irredundantSI}) is an irredundant irreducible decomposition.  
Then $\theta_{i'j'}^{-1}(S_{ij}) = \theta_{i'j'}^{-1}(S_{ij} \cap S_{i'j'})$ is a semialgebraic set with positive codimension in $\SO$.  
Since $\mu$ is absolutely continuous, it implies $\mu(\theta_{i'j'}^{-1}(S_{ij})) = 0$.  
\eqref{eq:ww} follows. \hfill $\qed$

\section{Proof of Theorem~\ref{thm:unique-determine-gaussians}}

We show how to reduce the proof to an application of Theorem~\ref{thm:unique-determine}. We have
\begin{equation}
    R\cdot \Phi(\bx) = \sum_{i=1}^p w_ie^{-\frac{\|R\bx-\ba_i\|^2}{2\kappa^2}} = \sum_{i=1}^p w_ie^{-\frac{\|\bx-R^T\ba_i\|^2}{2\kappa^2}}.
\end{equation}
Writing $\bx = (x,y,z)$ gives the following expression for the projection images
\begin{equation}
\begin{split}
        I_R(x,y) &= \sum_{i=1}^p w_i e^{-\frac{\|(x,y)-\pi R^T\ba_i\|^2}{2\kappa^2}} \int_{-\infty}^\infty e^{-\frac{(z-\pi_z R^T\ba_i)}{2\kappa^2}} dz
        \\
        &= \sum_{i=1}^p \sqrt{2\pi}\kappa w_i e^{-\frac{\|(x,y)-\pi R^T\ba_i\|^2}{2\kappa^2}},
        \end{split}
\end{equation}
where $\pi_z(a_1,a_2,a_3) := a_3$ is the projection operator onto the last coordinate. The second moment $M_2^G$ can then be written as
\begin{equation}
\begin{split}
    &M_2^G((x_1,y_1),(x_2,y_2)) \\
    &= \!\! \sum_{i,j=1}^p 2\pi\kappa^2 w_iw_j \int_{\SO} \!\!\!\!\!\!\!\!\!\! e^{-\frac{\| (x_1,y_1) - \pi R^T \ba_i \|^2 +\| (x_2,y_2) - \pi R^T \ba_j \|^2 }{2\kappa^2}} d\mu(R)\\
    &= \!\! \sum_{i,j=1}^p 2\pi\kappa^2 w_iw_j \int_{\SO} \!\!\!\!\!\!\!\!\!\! e^{-\frac{\| (x_1,y_1) - \pi R \ba_i \|^2 +\| (x_2,y_2) - \pi R \ba_j \|^2 }{2\kappa^2}} d\mu(R)\\
    &= 2\pi \kappa^2 M_2 * (k \otimes k),
    \end{split}
\end{equation}
where the second equality used the fact that the Haar measure on $\SO$ is invariant to transpositions \cite[Theorem~4.36]{kirillov2008introduction}, and $M_2$ is the second moment in the model of Theorem~\ref{thm:unique-determine}. Since $k$ has non-vanishing Fourier-transform, this equation can be deconvolved to obtain $M_2$. By Theorem~\ref{thm:unique-determine}, $M_2$ determines the weights and atomic positions $(w_i,\ba_i)$ up to a joint rotation and reflection, which concludes the proof.

\hfill $\qed$

\section{Sample complexity: Proofs of Theorems~\ref{thm:modified} and Corollary~\ref{cor:sample}}
\subsection{Proof of Theorem~\ref{thm:modified}}

The proof is divided into several steps.

\medskip 

\noindent \textbf{{Step 0}:}
We state  a general fact about real analytic functions that we will use:
\emph{Let $H(\by, \bz) : \mathbb{R}^m \times \mathbb{R}^n \rightarrow \mathbb{R}$ be real analytic jointly in $(\by, \bz)$.  Let $\nu$ be a compactly supported and absolutely continuous measure on $\mathbb{R}^n$.  Then $\int H(\by, \bz) d\nu(\bz)$ is real-analytic in $\by$.}
This can be justified by appropriately differentiating under the integral sign, see \cite{591021}.

\medskip

\noindent \textbf{{Step 1}:} 
We introduce notation.  

First, define
\begin{equation} \label{eq:X-space}
X = ( B(0,r) \times [w_-, w_+] )^{\times p} \subseteq \mathbb{R}^{4p},
\end{equation}
where $B(0,r)$ is the $\ell^2$-ball of radius $r$ in $\mathbb{R}^3$ centered at $0$. 
 $X$ is the space of possible molecules.

Next, slightly modifying the notation in the main text, write
\begin{equation*}
    {M}_2^{[m]} : X \rightarrow  \mathbb{R}^{2^m \times 2^m } \! \otimes  \mathbb{R}^{2^m \times 2^m } 
\end{equation*}
for the map associating a molecule to its pixelated second moment (\eqref{eq:secondmoment-modified}) when $2^m \times 2^m$ pixels are used.  Explicitly,
\begin{align*}
& {M}_2^{[m]}(\left\{(\ba_i, w_i) \right\}) = \\
   & \Bigg\{ \int_{j_4 \tau}^{(j_4 +1) \tau} \! \int_{j_3 \tau}^{(j_3 +1) \tau} \! \int_{j_2 \tau}^{(j_2 +1) \tau} \! \int_{j_1 \tau}^{(j_1 +1) \tau} \! \int_{\SO}  \sum_{i,j=1}^p w_i w_j \delta_{\pi R \ba_i}(x_1, y_1) \delta_{\pi R \ba_j}(x_2, y_2) \\
   &\qquad \qquad \qquad \qquad \qquad \qquad \qquad \qquad \ast (k(x_1,y_1)k(x_2,y_2)) d\mu(R) dx_1 dy_1 dx_2 dy_2 
   \Bigg\},
 \end{align*}
for $j_1, j_2, j_3, j_4 \in \{-2^{m-1},   \ldots, 2^{m-1} - 1\}$.
Then, ${M}^{[m]}_2$ is a real analytic function by Step 0.  Indeed, the Gaussian kernel $k(x,y)$ is real analytic, so the above integrand is real analytic in all variables. 
(Also, integration over $\SO$ is replaced by integration against a compactly supported absolutely continuous measure on $\mathbb{R}^3$ if we parameterize $\SO$ with Euler angles.)

Thirdly, we put 
\begin{equation*}
    L^{[m]}: \mathbb{R}^{2^{m+1} \! \times 2^{m+1}} \! \otimes \mathbb{R}^{2^{m+1} \! \times 2^{m+1}}  \rightarrow \mathbb{R}^{2^{m} \! \times 2^{m}} \! \otimes \mathbb{R}^{2^{m} \! \times 2^{m}}  
\end{equation*}
for the obvious linear map which lowers the resolution of the second moment by a factor of two, i.e.   $L^{[m]}(t) = s $ where 
\begin{equation*}
s_{j_1, j_2, j_3, j_4} = \sum_{\gamma_4 \in \{0, 1\}} \sum_{\gamma_3 \in \{0, 1\}} \sum_{\gamma_2 \in \{0, 1\}} \sum_{\gamma_1 \in \{0, 1\}} t_{2j_1 + \gamma_1, 2j_2 + \gamma_2, 2j_3 + \gamma_3, 2j_4 + \gamma_4}
\end{equation*}
for $j_1, j_2, j_3, j_4 \in \{-2^{m-1}, \ldots, 2^{m-1}-1 \}$.  
For all $m$, it holds
\begin{equation} \label{eq:coarsen}
    {M}^{[m]}_2 = L^{[m]} \circ {M}_2^
    {[m+1]}.
\end{equation}

\medskip

\noindent \textbf{{Step 2}:}
We prove that a certain stabilization occurs as $m \rightarrow \infty$.

Consider $X^{\times 2} = (B(0,r) \times [w_{-}, w_{+}])^{\times p} \times (B(0,r) \times [w_{-}, w_{+}])^{\times p} \subseteq \mathbb{R}^{8p}$, where the variables on $\mathbb{R}^{8p}$ are $\{ (\ba_i, w_i)\}, \{ (\bb_i, v_i) \}$.
Regard $X^{\times 2}$ 
as a semianalytic set (i.e., a subset of Euclidean space locally defined by real analytic equations and inequalities).  
Let $\mathcal{O}(X^{\times 2})$ denote the ring of real analytic functions on $X$. 
Then $\mathcal{O}(X^{\times 2})$ is a Noetherian ring, because $X$ is compact \cite[Th\'eor\`eme I, 9]{frisch1967points}.  (Note that $X^{\times 2}$ automatically satisfies the Stein hypothesis in \emph{loc. cit.} since we are in the real case, see \cite{siu1969noetherianness}.). 

We define
\begin{equation*}
\begin{split}
    \mathcal{I}^{[m]} = \text{ ideal in } &\mathcal{O}(X^{\times 2}) \text{ generated by the } 2^{4m} \text{ coordinate functions of } \\
    &M_2^{[m]}(\{ (\ba_i, w_i) \}) - {M_2^{[m]}}(\{(\bb_i, v_i) \}).
    \end{split}
\end{equation*}
For all $m$, we have
\begin{equation*}
    \mathcal{I}^{[m+1]} \supseteq \mathcal{I}^{[m]} 
\end{equation*}
by \eqref{eq:coarsen}.
From Noetherianity, there exists $ m' = m'(p, r, w_+, w_-)$ such that 
\begin{equation*}
\mathcal{I}^{[m]} = \mathcal{I}^{[m']} \quad \,\,\, \forall  m \geq m'.
\end{equation*}
Thus the corresponding zero sets in $X^{\times 2}$ stabilize too:
\begin{equation*}
    \Big{\{} \big{(} \{ (\ba_i, w_i) \}, \{ (\bb_i, v_i) \}  \big{)}   : \,  {M}_2^{[m]}(\{ (\ba_i, w_i)\}) = {M}_2^{[m]}(\{ (\bb_i, v_i)\}) \Big{\}}  \subseteq  X^{\times 2}
\end{equation*}
is constant in $m$ if $m \geq m'$. Equivalently, for all  $\{ (\ba_i, w_i) \} \in X$ we have:
\begin{equation} \label{eq:fiber-constant}
   \text{the fiber } \, ({M}_2^{[m]})^{-1}\!\big{(} {M}_2^{[m]}(\{\ba_i, w_i \}) \big{)} \! \subseteq X \,\, \text{ is constant in } m \text{ if } m \geq m'.
\end{equation}

\medskip 
 
\noindent \textbf{{Step 3}:} We deduce an equality in which there is no pixelation. 

Specifically, fix $\{(\ba_i, w_i) \}, \{(\bb_i, v_i)\} \in X$
such that $\{ (\ba_i, w_i)\}$ satisfies \textbf{A1}-\textbf{A2} in the main text and 
\begin{equation} \label{eq:equal-assume}
{M}_2^{[m']}(\{ (\ba_i, w_i)\}) = {M}_2^{[m']}(\{(\bb_i, v_i)\})
\end{equation}
holds. We claim there is an equality between unpixelated (but still blurred) second moments:
\begin{equation} \label{eq:equal-blurred-second}
    M_2(\{(\ba_i, w_i)\})  \ast  (k \otimes k)   = \, M_2(\{(\bb_i, v_i)\})   \ast  (k \otimes k).
\end{equation}

To see this, note by Step 0 that both sides of \eqref{eq:equal-blurred-second} are real-valued real analytic functions on $\mathbb{R}^2 \times \mathbb{R}^2$.
From continuity, if they differ on $[-1,1]^2$
there must exist a product of sufficiently small pixels where their integrals differ, i.e. $m \geq m'$ and $j_1, j_2, j_3, j_4 \in \{-2^{m-1}, \ldots, 2^{m-1} - 1\}$ such that
\begin{equation} \label{eq:bad-pixel}
  M_2^{[m]}(\{(\ba_i, w_i)\})((j_1, j_2), (j_3, j_4)) \, \neq \, M_2^{[m]}(\{(\bb_i, v_i)\})((j_1, j_2), (j_3, j_4)).
  \end{equation}
  However, \eqref{eq:bad-pixel} contradicts \eqref{eq:equal-assume} and (\ref{eq:fiber-constant}).
  Thus, \eqref{eq:equal-blurred-second} holds on $[-1,1]^2 \times [-1,1]^2$.  By real analyticity, \eqref{eq:equal-blurred-second} then holds on all of $\mathbb{R}^2 \times \mathbb{R}^2$ as wanted.
  
 \medskip

 \noindent \textbf{{Step 4}:} We undo the Gaussian blurring.
 
Continue with \eqref{eq:equal-blurred-second}.
 Because $M_2(\cdot)$ is compactly supported, it identifies with a tempered distribution.   
The Fourier transform is thus applicable to \eqref{eq:equal-blurred-second}.  
By the convolution theorem \cite[Thm.~7.1.15]{hormander2015analysis}, it gives
 \begin{equation} \label{eq:fourier}
      \widehat{M_2}(\{(\ba_i, w_i)\})    \, \widehat{k \otimes k}   \, \, \, = \, \, \widehat{M_2}(\{(\bb_i, v_i)\}   \, \widehat{k \otimes k}.
 \end{equation}
Note that the Paley-Wiener theorem \cite[Thm.~7.1.14]{hormander2015analysis} implies $\widehat{M_2}(\cdot)$ is a function rather just a distribution (moreover it is extendable to an entire function).
Also, $\widehat{k \otimes k}$ is a Gaussian function.  
Therefore, \eqref{eq:fourier} can be regarded as an equality of functions rather than just  distributions.
Since $\widehat{k \otimes k} \neq 0$ everywhere, it implies
\begin{equation*}
    \widehat{M_2}(\{(\ba_i, w_i)\})    \, \, = \, \, \widehat{M_2}(\{(\bb_i, v_i)\}),
\end{equation*}
whence
\begin{equation}
    M_2(\{(\ba_i, w_i)\}) = M_2(\{(\bb_i, v_i)\}),
\end{equation}
using the fact that the Fourier transform is an automorphism on tempered distributions.

\medskip

\noindent \textbf{Step 5:} We use Theorem~\ref{thm:theoretical} to conclude.

The above steps have shown: there exists $m'=m'(p,r,w_+, w_-)$ such that if $m \geq m'$ then $\{(\ba_i, w_i)\}, \{(\bb_i, v_i)\} \in X$ and $M_2^{[m]}(\{ (\ba_i, w_i) \}) = M_2^{[m']}(\{(\bb_i, v_i)\})$ imply $M_2(\{(\ba_i, w_i)\}) = M_2(\{(\bb_i, v_i)\})$.  However by Theorem~\ref{thm:theoretical}, $M_2(\{(\ba_i, w_i)\}) = M_2(\{(\bb_i, v_i)\})$ implies $\{(\ba_i, w_i)\}$ and $\{(\bb_i, v_i)\}$ are equal up to a rotation/reflection in $\mathbb{R}^3$, provided $\{(\ba_i, w_i)\}$ satisfies \textbf{A1}-\textbf{A2}.  

The proof of Theorem~\ref{thm:modified} is complete.  \hfill $\qed$

\subsection{Proof of Corollary~\ref{cor:sample}}
This now follows immediately from \cite[Sec.~3]{bandeira2017estimation} or \cite[Sec.~2]{pereira2019information},
because by Theorem~\ref{thm:modified} the second moment ${M_2}^{[m]}(\{(\ba_i,w_i)\})$ uniquely determines the signal $\{(\ba_i,w_i)\}$ up to  the group action of $O(3)$, provided $\{(\ba_i, w_i)\}$ satisfies the Zariski-open conditions \textbf{A1}-\textbf{A2}. 
\hfill $\qed$

\section{Normalized Bessel functions and the Nyquist criterion}
The spherical Bessel basis defined in the main text uses the normalized spherical Bessel functions $j_{\ell s}(k)$ defined by
\begin{equation}
    j_{\ell s}(k) = \frac{1}{c\sqrt{\pi}|j_{\ell+1}(R_{\ell,s})|}j_\ell(R_{\ell,s}\frac{k}{c}),
\end{equation}
where $j_\ell$ is the $\ell$th spherical Bessel function of the first kind \cite[Eq.~10.2.1]{NIST:DLMF}, $c$ the bandlimit of the projection images and $R_{\ell,s}$ the $s$th positive solution to $j_\ell = 0$. The Nyquist criterion determines the maximally allowable value of the truncation parameter $S_\ell$ by defining $S_\ell$ as the largest integer $s$ satisfying,
\begin{equation}
    R_{\ell,s+1} \leq 2\pi c R,
\end{equation}
assuming the projection images are supported on a disk of radius $R$~\cite{bhamre2017anisotropic}. Our numerical experiments used $c=0.5$ and a radius $R$ of $32$ voxels.

\bibliographystyle{plain}
\bibliography{references}

\end{document}